\documentclass{amsart}
\usepackage{lineno,hyperref}
\usepackage{amssymb}
\usepackage{subcaption}
\usepackage{multirow}
\usepackage{amsmath}
\usepackage{graphicx}
\usepackage{enumerate}
\usepackage{xcolor}
\usepackage{multirow}
\usepackage{commath}
\usepackage{tabularx}
\usepackage[top=0.75in, bottom=0.75in, left=0.875in, right=0.875in]{geometry}

\usepackage{tikz}
\usepackage{amsfonts,amssymb,amsmath,amsthm}
\usepackage{hyperref}
\usepackage{breqn}
\usepackage{calligra}
\usepackage{caption}
\usepackage{mathtools}
\usepackage{amsmath}
\usepackage{cleveref}
\usepackage{tikz}
\usepackage{multirow}
\usepackage{cleveref}
\usepackage{comment}
\usetikzlibrary{shapes,shadows,arrows}

\newtheorem{theorem}{Theorem}[section]

\newtheorem{lemma}[theorem]{Lemma}

\theoremstyle{definition}

\newtheorem{rem}{Remark}[section]

\theoremstyle{definition}

\theoremstyle{definition}

\theoremstyle{remark}

\usepackage{mathtools}

\theoremstyle{remark}

\usepackage{comment}
\usepackage{color}
\usepackage{tabularx}
\usepackage{array}
\usepackage{multirow}

\def\be{\begin{equation}}\def\ee{\end{equation}}
\def\ba{\begin{array}}\def\ea{\end{array}}
\def\bfg{\begin{figure}}\def\efg{\end{figure}}
\def\2pth{two point Hermite polynomial}
\def\c01{\mathcal{C}^{r}([0, 1])}

\def\01{[0, 1]}

\begin{document}

\title{A modified FC-Gram approximation algorithm with provable error bounds}

\author{Akash Anand}
\address{Akash Anand, Department of Mathematics and Statistics, Indian Institute of Technology Kanpur, India }
\email{akasha@iitk.ac.in}
\author{Prakash Nainwal}
\address{Prakash Nainwal, Department of Mathematics and Statistics, Indian Institute of Technology Kanpur, India }
\email{pnainwal@iitk.ac.in}


\begin{abstract}
The FC-Gram trigonometric polynomial approximation of a non-periodic function that interpolates the function on equispaced grids was introduced in 2010 by Bruno and Lyon  \cite{bruno2010high}. Since then, the approximation algorithm and its further refinements have 
been used extensively in numerical solutions of various PDE-based problems, and it has had impressive success in handling challenging configurations. While much computational evidence exists in the literature confirming the rapid convergence of FC-Gram approximations, a theoretical convergence analysis has remained open. In this paper, we study a modified FC-Gram algorithm where the implicit least-squares-based periodic extensions of the Gram polynomials are replaced with an explicit extension utilizing two-point Hermite polynomials. This modification brings in two significant advantages - (i) as the extensions are known explicitly, the need to use computationally expensive precomputed extension data is eliminated, which, in turn, facilitates seamlessly changing the extension length, and (ii) allows for establishing provable error bounds for the modified approximations. We show that the numerical convergence rates are consistent with those predicted by the theory through a variety of computational experiments.
\end{abstract}




\maketitle
 

 \vspace*{0.05in}

 

 \section{Introduction}
In many practical problems, the function to be approximated has discontinuous periodic extensions. For such functions, it is known that approximations by trigonometric polynomials do not converge uniformly and suffer from spurious oscillations near the boundary, an effect widely known as Gibb's phenomenon. Since its discovery,
numerous studies have been conducted to understand the underlying causes of the Gibbs phenomenon and to develop effective strategies to mitigate its effects (\cite{hewitt1979gibbs}, \cite{gottlieb1997gibbs}). One of the strategies to eliminate the ringing effect resulting from Gibbs' phenomenon relies on filtering out oscillations \cite{tadmor2007filters}. An alternative approach involves reconstructing a non-periodic function from its truncated Fourier series by expanding the series using Gegenbauer polynomials \cite{gottlieb1997gibbs, gelb2006robust, gottlieb1992gibbs}. The Padé approximations have also been used in this context \cite{driscoll2001pade, geer1995rational}.

In recent years, significant attention has been directed towards methods that rely on extending the function to a larger interval so that the extended function has a smooth periodic extension. In literature, such methods are typically called Fourier continuation (FC) or Fourier extension methods. Several researchers ( \cite{lyon2011fast}, \cite{matthysen2016fast}, \cite{huybrechs2010fourier}, \cite{anand2019fourier} \cite{boyd2002comparison}, \cite{boyd2001chebyshev}, \cite{bruno2007accurate}, \cite{gruberger2021two}), to name a few, have employed this idea to tackle a wide range of problems. 
Among these, the FC-Gram approximation strategy has garnered a lot of attention due to its success in handling a wide range of problems of practical importance.
The initial version of the technique was introduced by Bruno and Lyon in the year 2010 \cite{bruno2010high}. The method
has been employed successfully in solving a wide range of PDE-based problems involving complex geometries \cite{albin2011spectral, bruno2013, bruno2014, bruno2016, amlani2016fc, bruno2017, gaggioli2019, fontana2020, gaggioli2021, gaggioli2022, bruno2022, brunopaul2022, fontana2022}. Several innovations have been introduced in the continuation strategy of FC-Gram. Among others, one crucial design decision is to keep a fixed number of grid points in the extension interval. This fixed extension grid allows specific high-precision calculations essential for constructing an extension to be done offline and stored for repeated use later. Moreover, the size of the FFT calculations does not escalate rapidly with refinements in the mesh. Both these factors render the overall strategy computationally very efficient.
While the computational evidence in the literature confirms the rapid convergence of FC-Gram approximations, a theoretical convergence analysis has remained open.
In this paper, we study a modified FC-Gram algorithm where the implicit least-squares-based periodic extensions of gram polynomials are replaced with an explicit extension utilizing two-point Hermite polynomials.
This change is motivated by the following two factors. Firstly, the analytic expression for the extensions eliminates the need for computationally expensive precomputed extension data. This, in turn, facilitates refining the grid in the extension interval in a seamless manner to achieve convergence. Secondly, and more importantly, this modification allows for establishing provable error bounds for the approximations.

The paper is organized as follows: In section \ref{sec:fcgm}, we recall a version of the classical FC-Gram in some detail. Following this, in section \ref{sec:modfc}, we discuss the proposed modification in the FC-Gram algorithm. We also present a convergence analysis of the modified scheme. Finally, in section \ref{sec: numerics}, numerical experiments are included to demonstrate the performance of the approximation method and study the numerical and theoretical convergence rates.

\section{FC-Gram methodology} \label{sec:fcgm}

For $n\in\mathbb{N}$, consider the uniform grid on $[0,1]$ with grid spacing $h = 1/n$, that is, $x_j=jh, ~ 0 \le j \le n$. For given $C\in\mathbb{N}$ denote $b = 1+(C+1)h$. Given a smooth function $f$ on $[0,1]$, the FC-Gram algorithm constructs an interpolating trigonometric polynomial $t_{n,C}(f)$ of degree $n+C$ satisfying $t_{n,C}(f)(x_j) = f(x_j)=:f_j,\ j = 0, \ldots, n$, in two steps: extension and approximation. Both these steps can be summarized as follows:

\subsection{Extension}\label{subsec: extension}

 Construct a periodic function $f^c$ of period $b$ by introducing an extension $p$ on $[1,b]$ such that
    \begin{align}\label{f_c}
        f^c(x) =
        \begin{cases}
            f(x), & x\in[0, 1], \\
            p(x), & x\in[1, b].
        \end{cases}
    \end{align}
 The extension $p$ is constructed in the following two steps:

    \begin{enumerate}[i.]
        \item 
        {\bf (Projection onto Gram polynomials)}  
        For a chosen $s,d\in\mathbb{N}$ with $s(d-1)\le n+1$, two bilinear maps, given by
        \begin{align*}
            &\langle p,q\rangle_{L} = \sum_{j=0}^{d-1} p(x_{sj})q(x_{sj}), \\
            &\langle p,q\rangle_{R} = \sum_{j=0}^{d-1} p(x_{n-sj})q(x_{n-sj}),
        \end{align*}
        are used to obtain left and right Gram polynomials $p_\ell^L,\ p_\ell^R \in P_\ell$, $\ell = 0, \ldots, d-1$,  respectively satisfying $\langle p_k^L,p_\ell^L \rangle_{L} = \delta_{k\ell}$ and $\langle p_k^R,p_\ell^R \rangle_{R} = \delta_{k\ell}$ where $\mathcal{P}_{\ell}$ denotes the space of polynomials of degree $\ell$ or less. The parameter $s$ is referred to as the skipping parameter.
        Among the left Gram polynomials, the constant polynomial with $\langle p_0^L,p_0^L \rangle_{L} = 1$ is $p_0^L(x) = 1/\sqrt{d}$.  The linear polynomial satisfying $\langle p_1^L,p_0^L \rangle_{L} = 0$ and $\langle p_1^L,p_1^L \rangle_{L} = 1$ is given as
        \begin{align*}
        p_1^L(x) =  \sqrt{\frac{12}{(d-1)d(d+1)}} \left( \frac{x}{sh} -  \frac{d-1}{2} \right),\quad d\geq 2.
        \end{align*}
        It can be easily verified that the quadratic polynomial satisfying $\langle p_2^L,p_0^L \rangle_{L} = 0$, $\langle p_2^L,p_1^L \rangle_{L} = 0$ and $\langle p_2^L,p_2^L \rangle_{L} = 1$ is given by

         \begin{align*}
        p_2^L(x)=  \sqrt{\frac{5}{d(d^2-1)(d^2 - 4)}}\left( \frac{6 x^2}{s^2h^2} - \frac{6(d-1)x}{s h}+(d-1)(d-2)\right),\quad d\geq 3.
        \end{align*}
    
        The higher-degree left Gram polynomials are obtained following the Gram-Schmidt orthogonalization process 
        \begin{align}\label{Gram_Schmidt}
            q_\ell^L(x) = x^\ell - s^{\ell}h^\ell\sum_{k=0}^{\ell-1} \left(\sum_{j=0}^{d-1} j^\ell p_k^L(jsh)\right)p_k^L(x), \quad p_\ell^L(x) = q_\ell^L(x)/\sqrt{\langle q_\ell^L,q_\ell^L\rangle_{L}}.
        \end{align}
        The right Gram polynomials are obtained from the left Gram polynomials using 
        \begin{align*}
            p_\ell^R(x) = p_\ell^L(x-1+(d-1)sh), 
        \end{align*}
        a relation that follows from the fact that
        \begin{align*}
            &\langle p(\cdot-1+(d-1)sh),q(\cdot-1+(d-1)sh)\rangle_{R} =
            \langle p, q \rangle_L.
        \end{align*}
       
        The left and right Gram polynomials of degrees $5$ and less are shown in \Cref{fig:p} where $d=6$ and $s=1$. A straightforward induction on the degree $\ell$ shows that all these Gram polynomials are functions of $x/(sh)$, that is,
        \begin{align} \label{grampoly}
            p_\ell^L(x) = \tilde{p}_\ell^L(x/(sh)), \quad p_\ell^R(x) = \tilde{p}_\ell^R(x/(sh))
        \end{align}
        where $\tilde{p}_\ell^L,\ \tilde{p}_\ell^R$ are polynomials of degree $\ell$ with coefficients that do not depend on $sh$. For instance,
        \begin{align*}
            &\tilde{p}_1^L(u) = \sqrt{\frac{12}{(d-1)d(d+1)}} \left( u -  \frac{d-1}{2} \right),\\
            &\tilde{p}_2^L(u) = \sqrt{\frac{5}{d(d^2-1)(d^2-4)}}\left(6 u^2 - 6(d-1)u+(d-1)(d-2)\right).
        \end{align*}
        
        \begin{figure}
\centering
\begin{subfigure}{0.32\textwidth}
    \includegraphics[width=\textwidth,trim={35 20 35 10},clip]{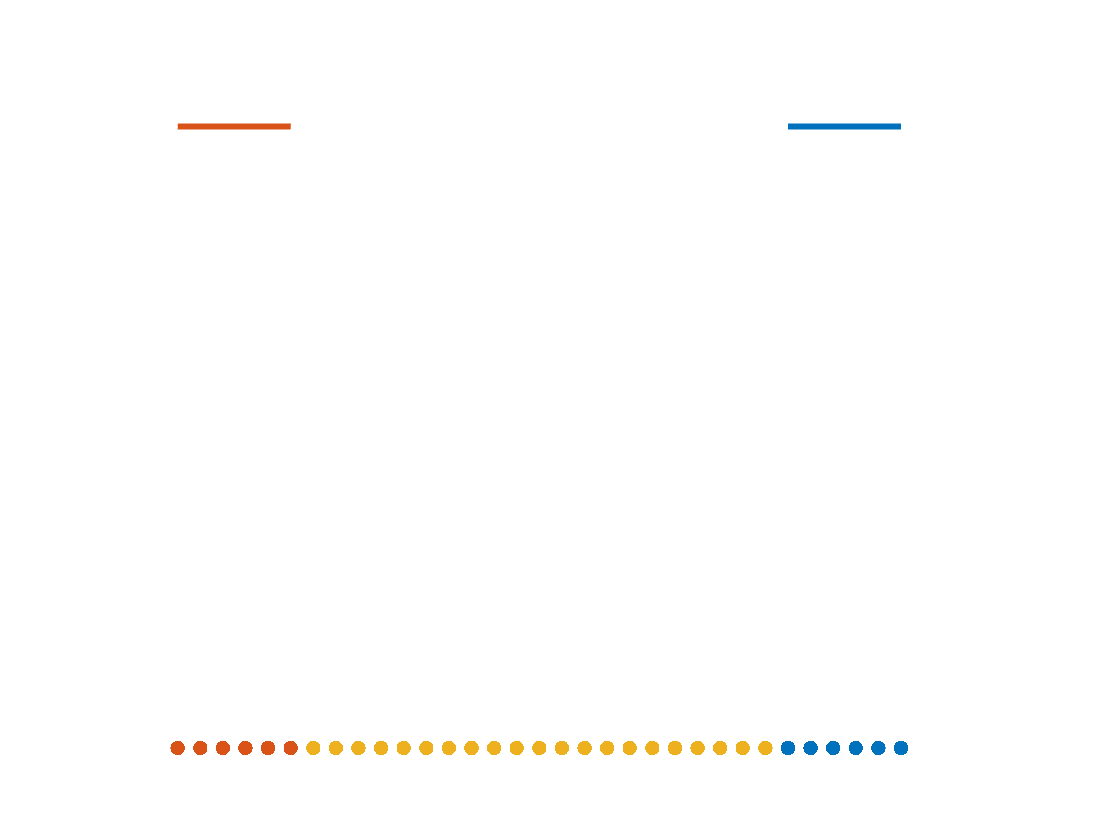}
    \caption{$p_0^L$ and $p_0^R$}
    \label{fig:p0}
\end{subfigure}
\hfill
\begin{subfigure}{0.32\textwidth}
    \includegraphics[width=\textwidth,trim={35 20 35 10},clip]{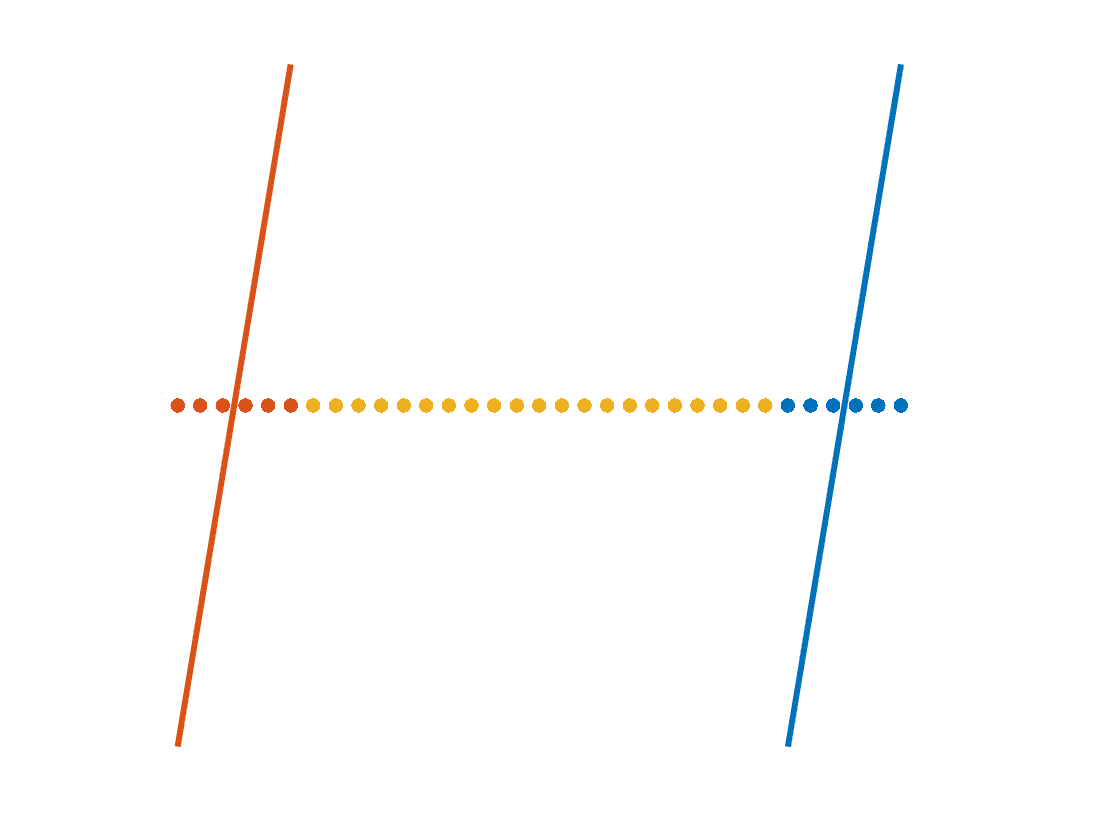}
    \caption{$p_1^L$ and $p_1^R$}
    \label{fig:p1}
\end{subfigure}
\hfill
\centering
\begin{subfigure}{0.32\textwidth}
    \includegraphics[width=\textwidth,trim={35 20 35 10},clip]{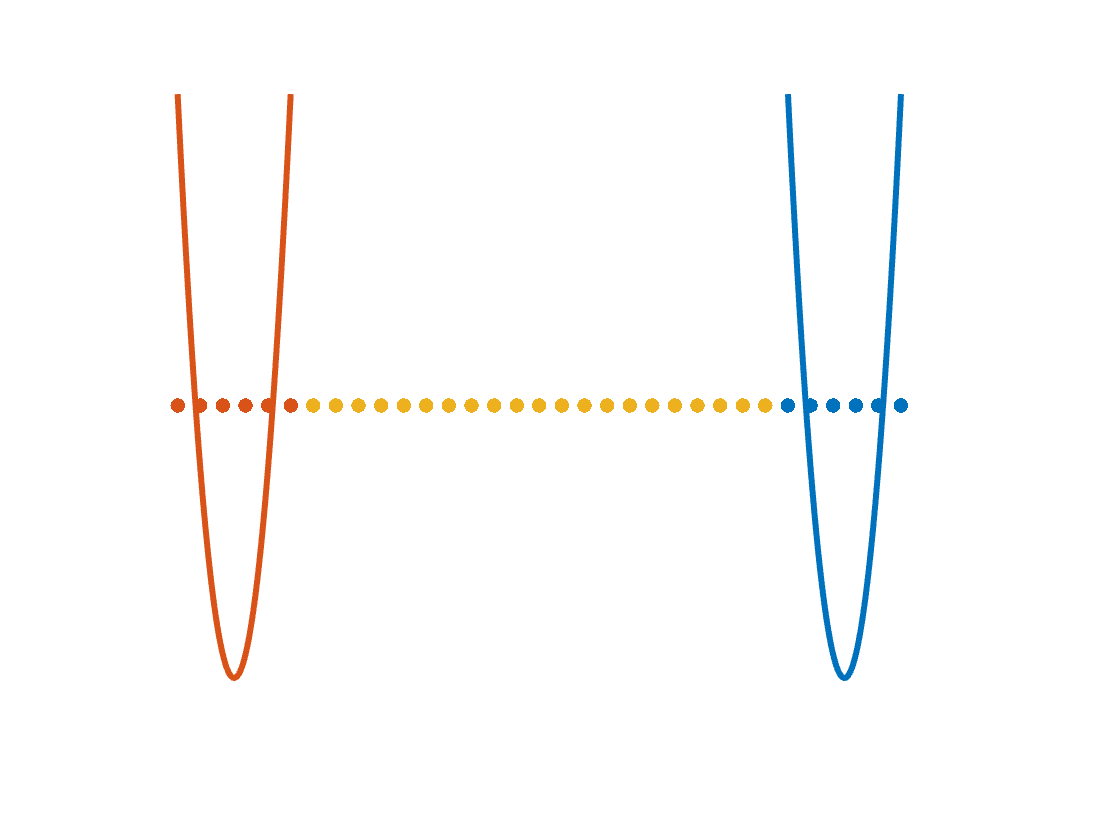}
    \caption{$p_2^L$ and $p_2^R$}
    \label{fig:p2}
\end{subfigure}
\vfill
\begin{subfigure}{0.32\textwidth}
    \includegraphics[width=\textwidth,trim={35 20 35 10},clip]{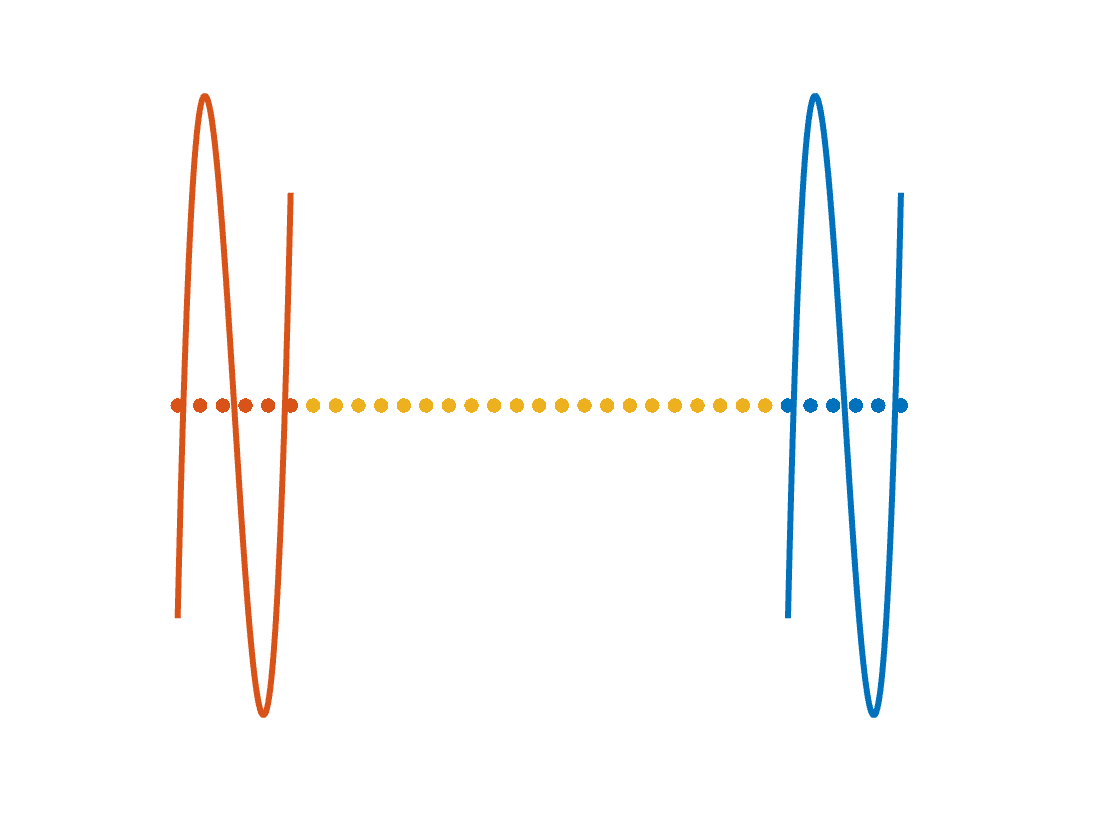}
    \caption{$p_3^L$ and $p_3^R$}
    \label{fig:p3}
\end{subfigure}
   \hfill
\centering
\begin{subfigure}{0.32\textwidth}
    \includegraphics[width=\textwidth,trim={35 20 35 10},clip]{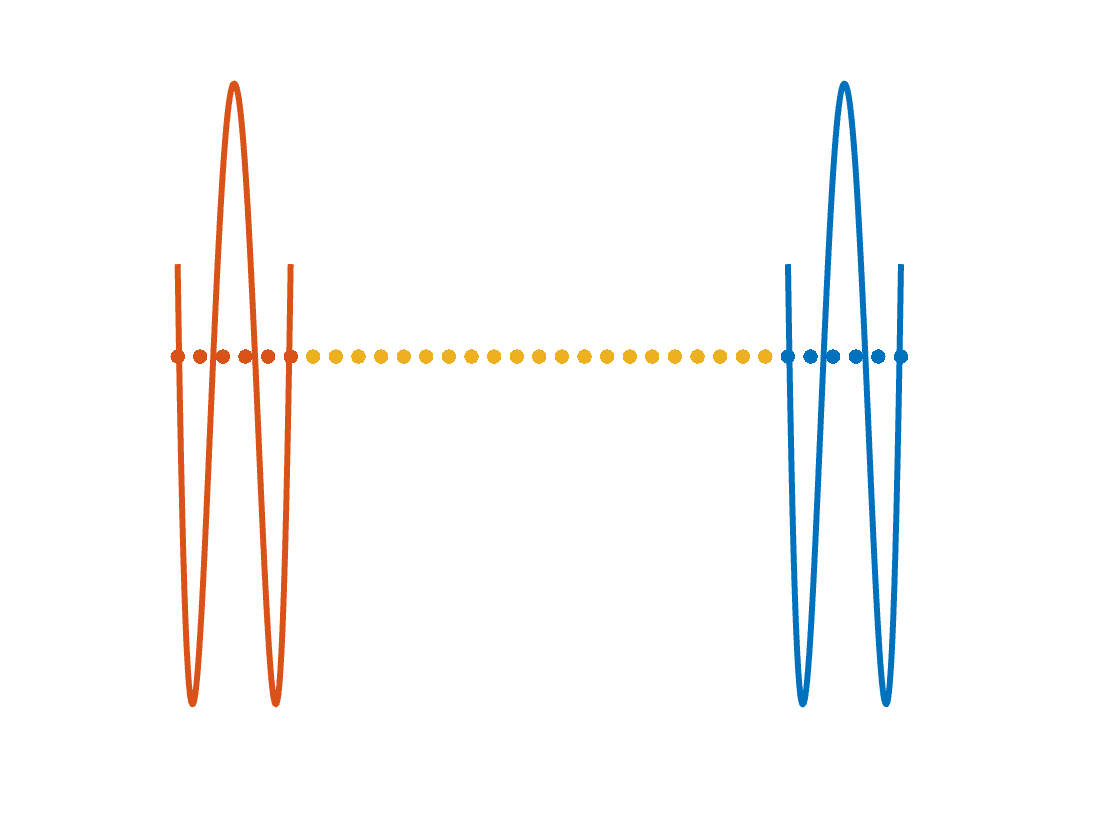}
    \caption{$p_4^L$ and $p_4^R$}
    \label{fig:p4}
\end{subfigure}
\hfill
\begin{subfigure}{0.32\textwidth}
    \includegraphics[width=\textwidth,trim={35 20 35 10},clip]{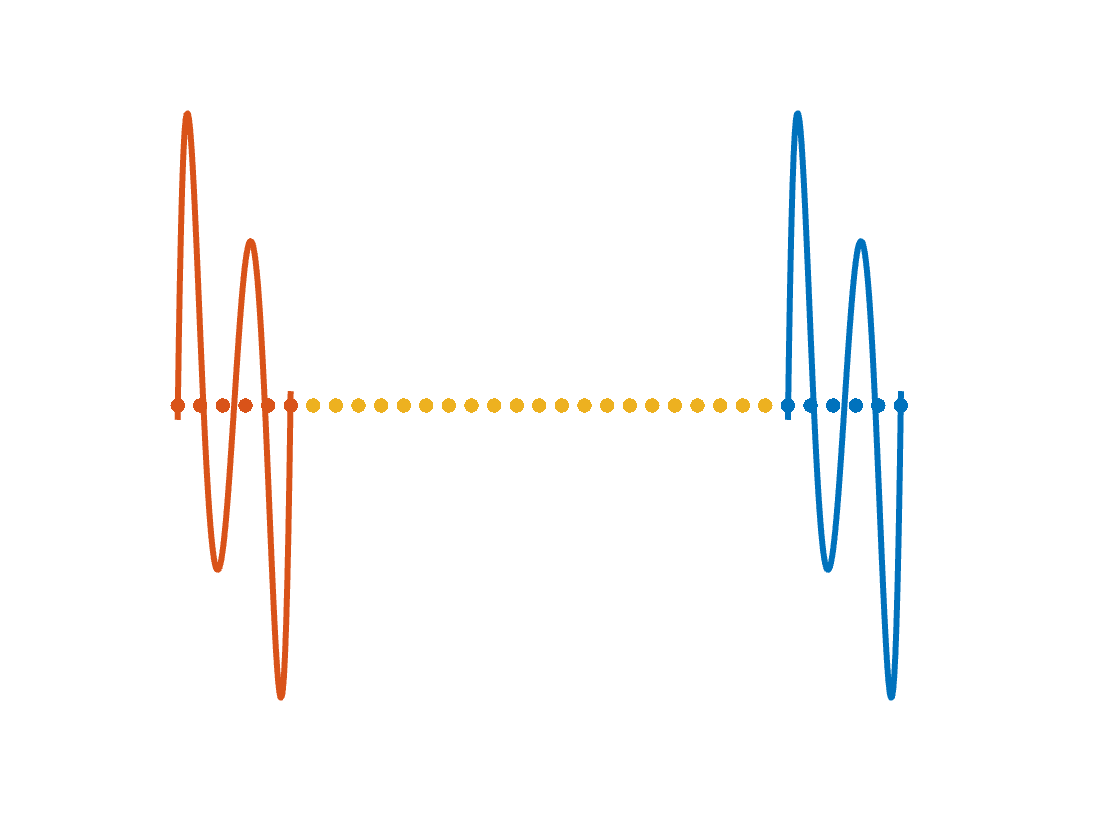}
    \caption{$p_5^L$ and $p_5^R$}
    \label{fig:p5}
\end{subfigure}   
\caption{Left and right Gram polynomials shown on the grid with $n = 32$, $d=6$ and $s=1$}
\label{fig:p}
\end{figure}


The left and right projection operators $P_{d}^L$ and $P_{d}^R$ respectively that project continuous functions on $[0,1]$ onto the space of polynomials of degree $d-1$ or less are defined as 
\begin{align}\label{ProjectionOnGramPoly}
   P_{d}^L(f)(x) =  \sum_{\ell = 0}^{d-1} \langle f,p_\ell^L \rangle_{L}\, p_\ell^L(x), \quad 
    P_{d}^R(f)(x) =  \sum_{\ell = 0}^{d-1} \langle f,p_\ell^R \rangle_{R}\, p_\ell^R(x).
\end{align}
        \item 
        {\bf (Continuation of Gram polynomials)}
        Corresponding to each Gram polynomial $p_\ell^L, \ p_\ell^R$, a periodic function is constructed satisfying 
        \begin{align*}
            p^{L,e}_\ell(x) \approx
            \begin{cases}
                0, &  1-(d-1)sh \le x \le 1, \\
                p_\ell^L(x - 1 - (C+1)h), &  1+(C+1)h \le x \le 1+(d+C)sh,
            \end{cases}
        \end{align*}
        and
        \begin{align*}
            p^{R,e}_\ell(x) \approx
            \begin{cases}
                p_\ell^R(x), &  1-(d-1)sh \le x \le 1, \\
                0, & 1+(C+1)h \le x \le 1+(d+C)sh,
            \end{cases}
        \end{align*}
 respectively. Note that one can obtain $p_{\ell}^{L, e}$ and $p_{\ell}^{R, e}$ in several different ways. Finally, the extension $p$ of $f$ is obtained as
        \begin{align}\label{left_right_ext}
            p(x) &= \sum_{\ell = 0}^{d-1} \langle f,p_\ell^L \rangle_{L}\, p^{L,e}_\ell\left(x\right) + \sum_{\ell = 0}^{d-1} \langle f,p_\ell^R \rangle_{R}\, p^{R, e}_\ell\left(x\right).
        \end{align}
    \end{enumerate}


\subsection{Approximation}

Once the extended data
\begin{align*}
    f^c_j&= \begin{cases}
       f(x_j), & j = 0,\cdots,n \\
       p(x_j), & j = n+1, \cdots, n+C
    \end{cases}
\end{align*}
using $p$ is available on the uniform grid of size $n+C$ on the interval $[0,b]$, the interpolating trigonometric polynomial $(t_{n, C}f)(x)$ satisfying
        \begin{align*}
            t_{n,C}(f)(x_j) = f^c_j, \quad j = 0, \ldots, n, n+1, \ldots, n+C.
        \end{align*}
is obtained as
\begin{align} \label{eq:tnc}
   t_{n, C}(f)(x)& = \sum_{j=0}^{n+C} f^c_j L_j^{(n+C)}(x),
\end{align}
where, 
\begin{align*}
     L_j^{(n+C)}(x) &= \begin{cases}
    1, & x = x_j, \\
    \dfrac{1}{n+C+1}\left(\sin \left(\dfrac{\pi  (n+C) (j-n x)}{n+C+1}\right) \csc \left(\dfrac{\pi 
   (j-n x)}{n+C+1}\right)+\cos (\pi  (j-n x))\right), & x \ne x_j,
    \end{cases}
\end{align*}
provided $n+C$ is odd, that we assume henceforth.

\subsection{Continuation of Gram polynomials in FC-Gram} \label{phi_LS}

\begin{figure}
\centering
    \includegraphics[width=0.8\textwidth]{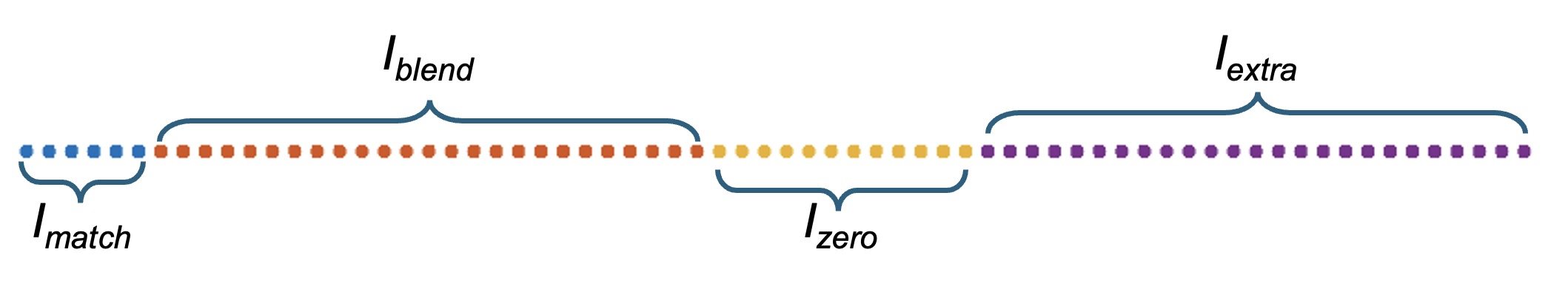}
    \caption{The uniform grid used for the construction of Gram polynomial extensions where $d = 6,~ C = 25,~ Z = 12$ and $E = 25$}
    \label{fig:grid}
\end{figure}

We now briefly recall the construction of $p_\ell^{L,e}$ and $p_\ell^{R,e}$ adopted by the classical FC-Gram (see \cite{amlani2016fc} for further details). For obtaining the extension of Gram polynomials, the following grids are used (see \Cref{fig:grid}):

\begin{align*}
    &I_{\text{match}} = \{ jsh : j = 0, \ldots, d-1 \},\\
    &I_{\text{blend}}=  \{ (d-1)sh + jh : j = 1, \ldots, C \},\\
    &I_{\text{zero}} = \{ ((d-1)s+C)h+ jsh : j = 1, \ldots, Z\},\\
    &I_{\text{extra}} = \{ ((d+Z-1)s+C)h + jh : j = 1, \ldots, E \},
\end{align*}
where integers $Z\ge d$ and $e_n$ are parameters of the construction algorithm. 
The extension is sought in the form of a trigonometric polynomial
\begin{align*}
    \varphi_\ell^{\text{LS}}(x)&= \sum_{k =-M}^{M} a_k^\ell \exp\left(\frac{2\pi i k x}{((d+Z-1)s+C+E+1)h}\right)
\end{align*}
where the unknown coefficients $a_k^\ell$ are obtained by solving an over-determined system of linear equations (in the least square sense). To obtain the desired system of linear equations, the refined grids $I_{\text{match}}^{\text{\text{ref}}}$ and  $I_{\text{zero}}^{\text{\text{ref}}}$ are obtained by inserting grid points in $I_{\text{match}}$ and $I_{\text{zero}}$ respectively with grid spacing $h/n_{\text{over}}$, that is,
\begin{align*}
    &I_{\text{match}}^{\text{\text{ref}}} = \{ jsh/n_{\text{over}} : j = 0, \ldots, (d-1)n_{\text{over}} \},\\
    &I_{\text{zero}}^{\text{\text{ref}}} = \{ ((d-1+j)s+C)h/n_{\text{over}} : j = 1, \ldots, Z n_{\text{over}} \},
\end{align*}
where $n_{\text{over}}$ is a chosen refinement parameter.
Finally, the corresponding system of linear equation is set to enforce
\begin{align*}
  \varphi_\ell^{\text{LS}}( z_j ) = p_{\ell}^L(z_j), \quad z_j \in I_{\text{match}}^{\text{\text{ref}}},\quad\quad
  \varphi_\ell^{\text{LS}}( z_j) = 0, \quad z_j \in I_{\text{zero}}^{\text{\text{ref}}}.
\end{align*}
More precisely, the linear system reads
\begin{align} \label{linsys}
    \begin{bmatrix}
    A \\ B
    \end{bmatrix}
    \begin{bmatrix}
        a_{-M}^\ell \\ \vdots \\ a_{M}^\ell
    \end{bmatrix}
    = \begin{bmatrix}
        \alpha_\ell \\ \beta_\ell
    \end{bmatrix}
\end{align}
where $((d-1)n_{\text{over}}+1)\times(2M+1)$ matrix $A$, $((Z-1)n_{\text{over}}+1)\times(2M+1)$ matrix $B$, and corresponding right hand side vectors $\alpha$ and $\beta$ are given by

  \begin{align*}
    A &= 
       \begin{bmatrix}
    1 & \cdots & 1 \\
        \exp\left(\dfrac{-2\pi is M}{((d+Z-1)s+C+E+1)n_{\text{over}}}\right) & \cdots & \exp\left(\dfrac{2\pi i sM}{((d+Z-1)s+C+E+1)n_{\text{over}}}\right) \\ \vdots & \ddots & \vdots \\
        \exp\left(\dfrac{-2\pi i Ms (d-1)}{((d+Z-1)s+C+E+1)}\right) & \cdots & \exp\left(\dfrac{2\pi i Ms (d-1)}{((d+Z-1)s+C+E+1)}\right) \\
    \end{bmatrix}, \\
    B &= 
       \begin{bmatrix}
        \exp\left(\dfrac{-2\pi i M(ds+C)}{((d+Z-1)s+C+E+1)n_{\text{over}}}\right) & \cdots & \exp\left(\dfrac{2\pi i M(ds+C)}{((d+Z-1)s+C+E+1)n_{\text{over}}}\right) \\ \vdots & \ddots & \vdots \\
        \exp\left(\dfrac{-2\pi i M ( (d-1+Z)s+C )}{((d+Z-1)s+C+E+1)}\right) & \cdots & \exp\left(\dfrac{2\pi i M ((d-1+Z)s+C)}{((d+Z-1)s+C+E+1)}\right) \\
    \end{bmatrix}, \\
    \alpha_\ell &=
    \begin{bmatrix}
        p_\ell^L(0) \\ p_\ell^L(sh/n_{\text{over}}) \\ \vdots \\ p_\ell^L((d-1)sh)
    \end{bmatrix} =
    \begin{bmatrix}
        \tilde{p}_\ell^L(0) \\ \tilde{p}_\ell^L(1/n_{\text{over}}) \\ \vdots \\ \tilde{p}_\ell^L((d-1))
    \end{bmatrix},
    \quad\quad\quad
    \beta_\ell =
    \begin{bmatrix}
        0 \\ \vdots \\ 0
    \end{bmatrix}.
\end{align*}

For the choice $E=C$, the extensions $p_\ell^{L,e}$ and $p_\ell^{R,e}$ are obtained from $\varphi_\ell^{\text{LS}}$ using
\begin{align}
    &p_\ell^{R,e}(x) = \varphi_\ell^{\text{LS}}(x-1+ (d-1)sh),
    \label{eq:leftext} 
    \\ 
    &p_\ell^{L,e}(x) = \varphi_\ell^{\text{LS}}(x-1 + ((d-1+Z)s+C)h) .
    \label{eq:rightext}
\end{align}

\begin{rem} \label{linsysrem}
 The linear system \cref{linsys} is independent of $h$; therefore, $a_k^\ell$ can be pre-computed to high accuracy for each $s$ and saved for repeated use.
\end{rem}

\subsection{The FC-Gram approximation errors} \label{sec:uniform}
It is known \cite{lyon2010high} that, for a fixed $C$, the approximation error $e_n = \Vert t_{n,C}(f)-f\Vert_{\infty,[0,1]}$ does not go to zero as $n$ approaches infinity. In practice, an optimal fixed $C$ is chosen to achieve a near-machine precision approximation accuracy. To achieve full convergence, the authors, in \cite{lyon2010high}, suggest increasing the skipping parameter $s$ as $n$ increases. An example of non-convergence of FC-Gram approximations for fixed $s$ and $C$ values can be seen in the convergence study presented in \cref{with skipping}. The same table confirms convergence as $s$ increases with $n$ while $C$ is kept fixed. Alternatively, convergence can also be achieved by keeping $s$ fixed and increasing $C$ as $n$ increases as can be seen in the convergence study presented in \cref{with increasing C}. The numerical order of convergence ($noc_n$) is obtained as
$noc_n = \log_2 (e_{n/2}/e_{n})$.

Each of these approaches has its advantages and disadvantages. While increasing the skipping parameter necessitates the availability of a voluminous amount of computationally expensive precomputed extension data, the amount of online computation stays comparatively low due to the fixed number of grid points in the extension interval. On the other hand, in view of \cref{linsysrem}, keeping $s$ fixed reduces the cost of solving the linear system in \cref{linsys} drastically as matrices $A$ and $B$ do not change with $n$. At the same time, the cost of online computation is more as the size of FFT calculations grows with $n$. One must note, however, that the computational complexity of both these approaches is $O(n \log n)$. In the modified FC-Gram that we study in this paper, we choose the second option of increasing $C$ with $n$ while keeping the skipping parameter fixed at $s=1$.

\begin{table} [!t] 
\centering
	\begin{tabularx}{0.70\textwidth}{ >{\setlength\hsize{0.6\hsize}\centering}X | >{\setlength\hsize{1.3\hsize}\centering}X | >{\setlength\hsize{0.9\hsize}\centering}X | >{\setlength\hsize{1.3\hsize}\centering}X | >{\setlength\hsize{0.9\hsize}\centering}X  }
		\hline
		\multirow{2}{\hsize}{\centering$n$} &
		\multicolumn{2}{>{\setlength\hsize{2.2\hsize}\centering}X |}{$s= 1, \quad C = 25$} &  \multicolumn{2}{>{\setlength\hsize{2.2\hsize}\centering}X }{$s = n/8, \quad C = 25$}  
		\tabularnewline
		\cline{2-5}
		& $e_n$ & $noc_n$ &  $e_n$  & $noc_n$ 
            \tabularnewline
		\hline\hline
   $2^{3}$ & $ 8.31\times 10^{-12} $ & --- & $ 8.31\times 10^{-12} $ & ---	
          \tabularnewline
		\hline
     $2^{4}$ & $ 1.16\times 10^{-11} $ & $ -0.48 $ & $ 9.49\times 10^{-12} $ & $ -0.19 $    
            \tabularnewline
		\hline
    $2^{5}$ & $ 1.55\times 10^{-11} $ & $ -0.41 $ & $ 2.32\times 10^{-13} $ & $ 5.36 $   
              \tabularnewline
		\hline
    $2^{6}$ & $ 1.72\times 10^{-11} $ & $ -0.15 $ & $ 4.16\times 10^{-15} $ & $ 5.80 $  
              \tabularnewline
		\hline
	\end{tabularx}
	\caption{A convergence study for approximations of $ f(x) = 1$ with skipping $s = n/8$ and without skipping $s =1$. Here $d=6$ and for refinement grid we have taken $N= 2^{14}$ }
 \label{with skipping}
\end{table}

\begin{table}[bht!]

\centering
\begin{tabularx}{0.40\textwidth}{ >{\setlength\hsize{0.6\hsize}\centering}X | >{\setlength\hsize{1.4\hsize}\centering}X | 
>{\setlength\hsize{1.0\hsize}\centering}X }
\hline
	$n$ & $e_n$ & $noc_n$  \tabularnewline
    \hline	\hline
   $2^{3}$ & $ 2.25\times 10^{-01} $ & ---
   \tabularnewline 
    \hline
   $2^{4}$ & $ 1.60\times 10^{-03} $ & $ 7.13 $ 
   \tabularnewline 
    \hline
  $2^{5}$ & $ 4.29\times 10^{-05} $ & $ 5.22 $ 
    \tabularnewline
    \hline
   $2^{6}$ & $ 6.29\times 10^-{08} $ & $ 9.41 $
    \tabularnewline 
    \hline
  $2^{7}$ & $ 2.53\times 10^{-13} $ & $ 17.92 $
   \tabularnewline 
    \hline
  $2^{8}$ & $ 3.33\times 10^{-15} $ & $ 6.25 $
  \tabularnewline 
    \hline
\end{tabularx}
   \caption{Convergence study for approximations of $ f(x) = 1$ \text{with} $ C = n/4 - 1$}
\label{with increasing C}
 \end{table}

\section{A modified FC-Gram} \label{sec:modfc}

We now study the algorithm where the $\phi_\ell^{\text{LS}}$ used in the classical version of FC-Gram for the extension of Gram polynomials is replaced with, say $\phi_\ell^H$, whose construction we describe below.

\subsection{Continuation of Gram polynomial using Hermite polynomials} \label{phi_H}

Using the polynomials
\begin{align}
    p_m^{u_1, u_2}(u)&=\dfrac{1}{m!}(u-u_1)^m\left(\dfrac{u-u_2}{u_1-u_2} \right)^{r+1}\sum_{\ell = 0}^{r-m}\binom{r+\ell}{r}\left(\dfrac{u-u_1}{u_2-u_1} \right)^{\ell},\label{pm_0}
\end{align}
we define

\begin{align*}
  \varphi_{\ell}^{\text{H}}(x)&= \begin{cases}
      p_{\ell}^L(x), & x\in[0, (d-1)h]\\
      \sum\limits_{m = 0}^{\ell} \left( p_{\ell}^{L}\right)^{(m)}\left((d-1)h\right) p_{m}^{(d-1)h, (d+C)h}(x), & x\in [(d-1)h, (d+C)h]\\
      0,& x\in [(d+C)h, (d+C+Z-1)h]\\
      \sum\limits_{m = 0}^{\ell} \left( p_{\ell}^{L}\right)^{(m)}\left(0\right) p_{m}^{(d+C+Z+E)h,(d+C+Z-1)h}(x), & x\in [(d+C+Z-1)h, (d+C+Z+E)h].
  \end{cases}
\end{align*}
As an example, in \cref{fig:phi_LS}, we depict $\varphi_{\ell}^{\text{H}}$ and $\varphi_{\ell}^{\text{LS}}$ for $\ell=0,1,2,3$ with $s=1$,  $d= 5$, $C = 25$, $Z= 12$, $E = 25$ and $h = 1/256$. 

Again, as in \cref{eq:leftext} and \cref{eq:rightext}, we obtain the required extensions of right and left Gram polynomials as
\begin{align*}
    p_{\ell}^{R, e}(x) &= \varphi_{\ell}^{\text{H}}(x-1+ (d-1)h),\\
      p_{\ell}^{L, e}(x) &= \varphi_{\ell}^{\text{H}}(x-1+ (d+C+Z-1)h),
\end{align*}
for $0\leq \ell \leq d-1$, which, on further simplification, yields
\begin{align}
    p_{\ell}^{R, e}(x) &= 
    \sum\limits_{m = 0}^{\ell} \left( p_{\ell}^{R}\right)^{(m)}\left(1\right) p_{m}^{1,b}(x),\label{left_ext_Hermite}\\
   p_{\ell}^{L, e}(x) &=    
   \sum\limits_{m = 0}^{\ell} \left( p_{\ell}^{L}\right)^{(m)}\left(0\right) p_{m}^{b, 1}(x).\label{right_ext_Hermite}
\end{align}
Finally, we obtain corresponding $p$ as in \cref{left_right_ext} to get the modified continuation to $f$ given by \cref{f_c}.

\begin{figure}[t]
\centering
\begin{subfigure}{0.49\textwidth}
    \includegraphics[width=\textwidth,trim={105 140 105 60},clip]{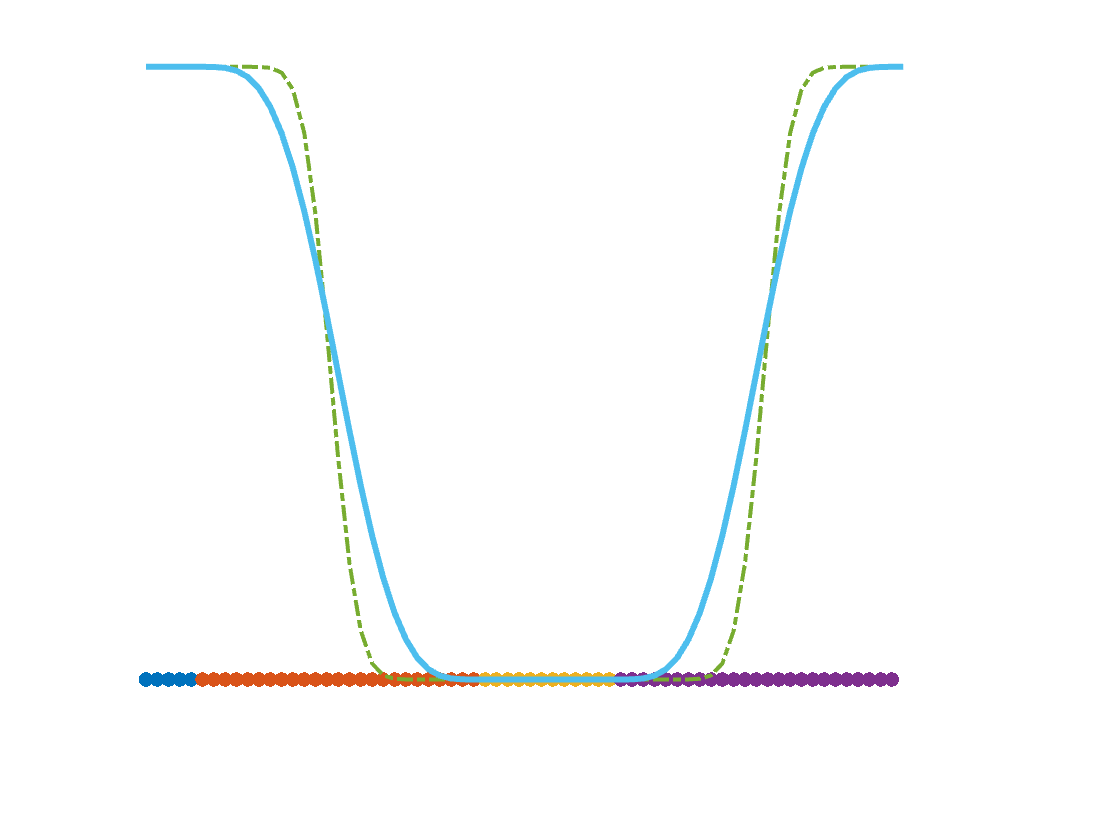}
    \caption{$\varphi_{0}^{\text{LS}}$ and $\varphi_{0}^{\text{H}}$}
    \label{fig:p0}
\end{subfigure}
\hfill
\begin{subfigure}{0.49\textwidth}
    \includegraphics[width=\textwidth,trim={105 140 105 120},clip]{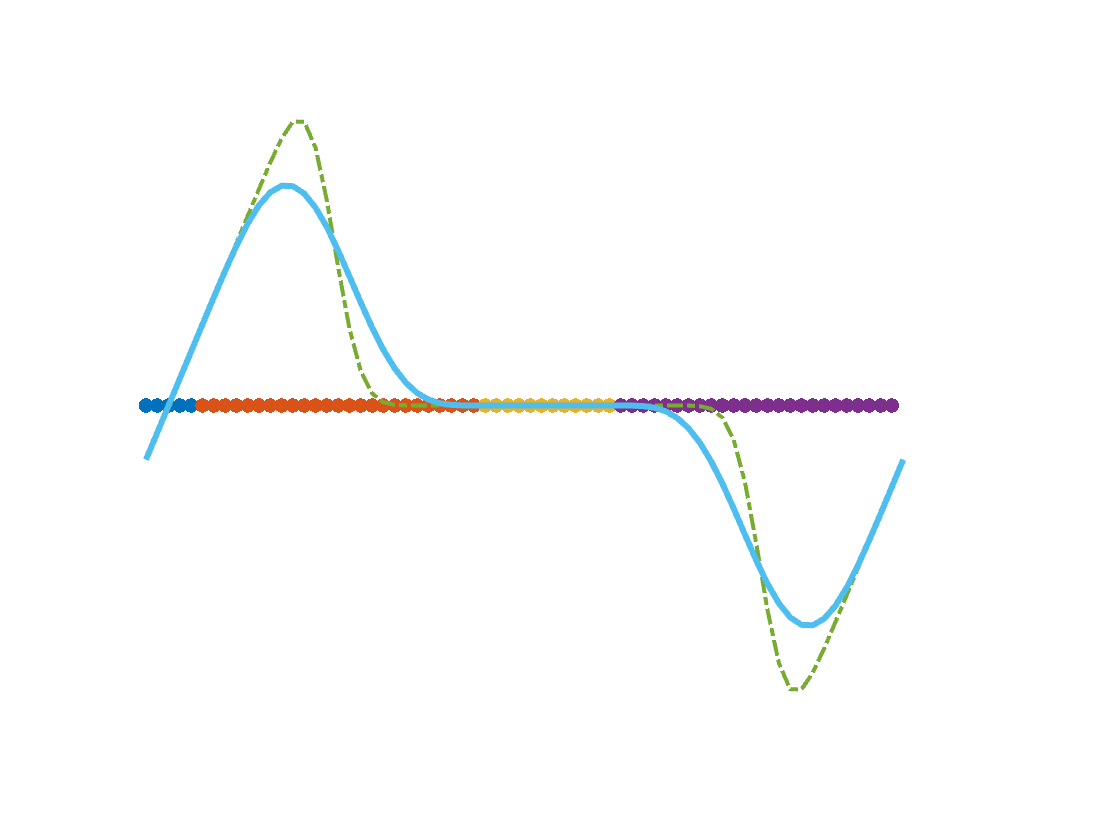}
    \caption{$\varphi_{1}^{\text{LS}}$ and $\varphi_{1}^{\text{H}}$}
    \label{fig:p1}
\end{subfigure}
\vfill
\centering
\begin{subfigure}{0.49\textwidth}
    \includegraphics[width=\textwidth,trim={105 140 105 130},clip]{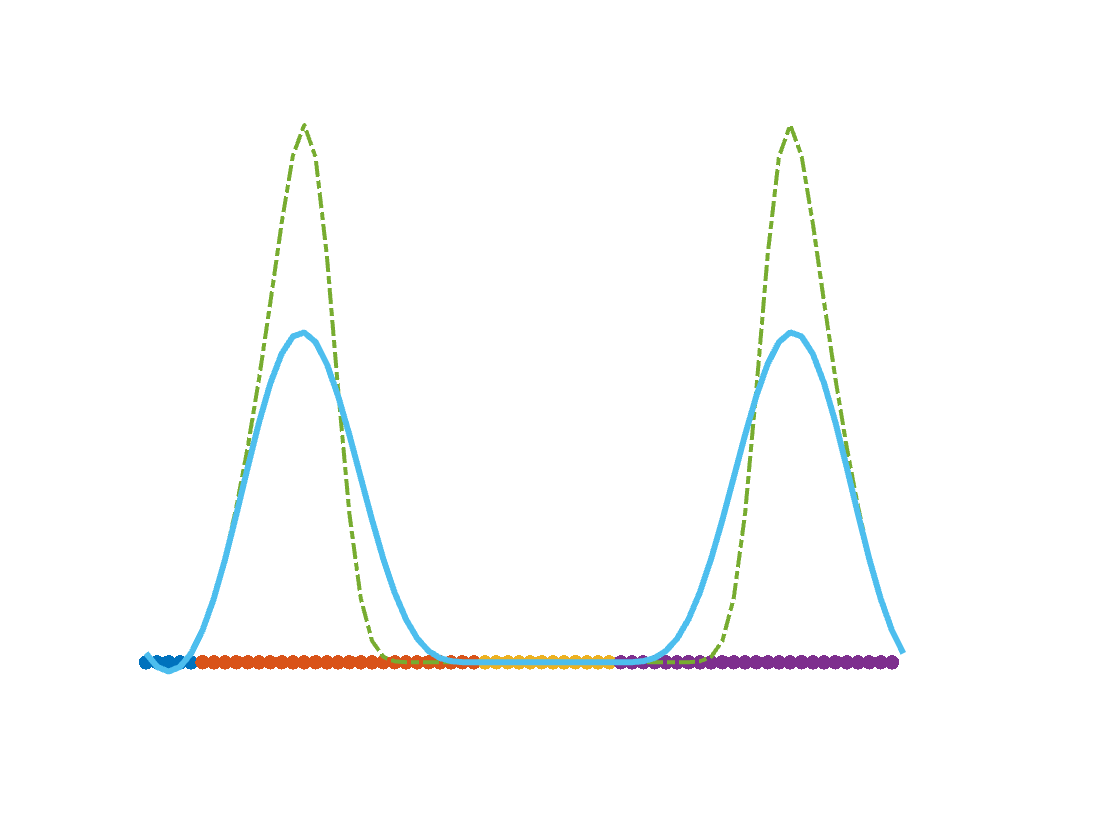}
    \caption{$\varphi_{2}^{\text{LS}}$ and $\varphi_{2}^{\text{H}}$}
    \label{fig:p2}
\end{subfigure}
\hfill
\begin{subfigure}{0.49\textwidth}
    \includegraphics[width=\textwidth,trim={105 120 105 60},clip]{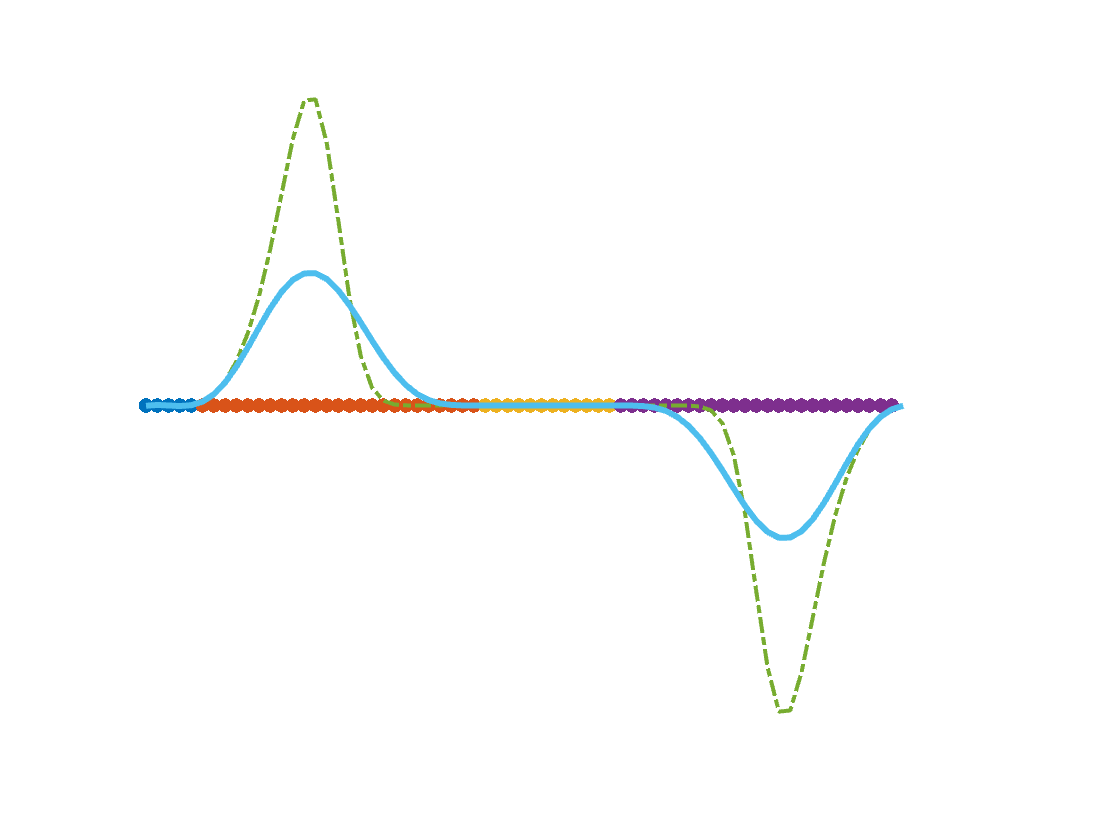}
    \caption{$\varphi_{3}^{\text{LS}}$ and $\varphi_{3}^{\text{H}}$}
    \label{fig:p3}
\end{subfigure}
\hfill  
\caption{plots of $\varphi_{\ell}^{\text{LS}}$ (dotted lines) and $\varphi_{\ell}^{\text{H}}$ (solid lines) for $0\leq \ell \leq 4$ with $s=1$,  $d= 5$, $C = 25$, $Z= 12$, $E = 25$ and $h = 1/256$}
\label{fig:phi_LS}
\end{figure}

\subsection{Convergence analysis}

In this section, we provide a convergence analysis of the proposed modified FC-Gram approximation. The following theorem states the main result establishing a bound for the approximation errors. As a consequence, the result also reveals the theoretical rate of convergence of the numerical scheme.

\begin{theorem}\label{theorem: main}
   Let $f\in C^{\infty}\left( [0, 1]\right)$ for some $d\in \mathbb{N}$. For $b\in\mathbb{Q}$ with $b>1$, let $\mathbb{N}_b = \{n\in\mathbb{N} : nb \in \mathbb{N} \text{ and } 2 \mid nb\}$. For all $n \in \mathbb{N}_b$, let $\tau_n(f) := t_{n,nb-n-1}$ be the trigonometric polynomial approximation that interpolates $f$ at $x_j = j/n$ where
    $t_{n,C}$ (as given in \cref{eq:tnc}) uses the extension $p$ (in \cref{left_right_ext}) with Gram polynomial extensions $p_{\ell}^{L,e}$, $p_{\ell}^{R,e}$ obtained using $\varphi_\ell^{\text{H}}$ (as in \cref{left_ext_Hermite} and \cref{right_ext_Hermite}). Then, there exists a $M>0$, independent of $n$, such that
    \begin{align*}
    &\Vert \tau_{n}(f)-f\Vert_{\infty,[0,1]} \le M n^{-d} 
    \end{align*}
    for all $n\in\mathbb{N}_b$.
\end{theorem}

To estimate the error in the numerical approximation method for a given parameter $d \in \mathbb{N}$, we make use of a reference extension, namely, $p^d_{\text{ref}}$ given by  
\begin{align} \label{p_ref}
    p_{\text{ref}}^d(x) = \sum_{m = 0}^{d-1} f^{(m)}(1) p_m^{1, b}(x) + \sum_{m = 0}^{d-1} f^{(m)}(0) p_m^{b, 1}(x).
\end{align}

Using this reference extension, we get
    \begin{align}\label{f_ref}
        f^c_{\text{ref}}(x) =
        \begin{cases}
            f(x), & x\in[0, 1], \\
            p^d_{\text{ref}}(x), & x\in[1, b].
        \end{cases}
    \end{align}
as another continuation of $f$. The approximation error can then be expressed as
\begin{align*}
& f(x) - \tau_{n}(f^c)(x)=
\left(f-f^c \right)(x) + \left[(f^c - f^c_{\text{ref}})(x) - \tau_{n}(f^c-f^c_{\text{ref}})(x)\right] + \left(f^c_{\text{ref}} - \tau_{n}(f^c_{\text{ref}})\right)(x).
\end{align*}
For $x\in [0,1]$, the expression for the error further simplifies as
\begin{align} \label{eq:error0}
    f(x) - \tau_{n}(f)(x) &= \left[(f^c - f^c_{\text{ref}})(x) - \tau_{n}(f^c-f^c_{\text{ref}})(x)\right] + \left(f^c_{\text{ref}} - \tau_{n}(f^c_{\text{ref}})\right)(x)
\end{align}
We, thus, see that the approximation error can be estimated in terms of $\left(g - \tau_{n}(g)\right)(x)$. Note that, using the $b$-periodic Fourier series of $g$ 
\begin{align*}
    g(x) = \sum_{\ell=-\infty}^{\infty} c_\ell(g)e^{2\pi i \ell x/b}, \quad c_\ell(g) = \frac{1}{b}\int_0^b g(x)e^{-2\pi i \ell x/b},
\end{align*}
and the trigonometric polynomial 
\begin{align*}
    g(x) = \sum_{\ell=-nb/2}^{(nb/2)-1} \tilde{c}_\ell(g)e^{2\pi i \ell x/b}, \quad \tilde{c}_\ell(g) = \frac{1}{nb}\sum_{j=0}^{nb-1} g(x_j)e^{-2\pi i \ell x_j/b},
\end{align*} 
that interpolates $g$ at $x=x_j, j = 0,\ldots, nb-1$,
we have
\begin{align*}
    g(x)-\tau_{n}(g)(x) 
    = \sum_{\ell=-nb/2}^{(nb/2) -1} \sum_{\substack{m=-\infty \\ m \ne 0}}^{\infty} c_{\ell+(nb)m}(g)e^{2\pi i (\ell+(nb)m) x/b} - \sum_{\ell=-nb/2}^{(nb/2) -1} \left( \tilde{c}_\ell(g) - c_\ell(g)\right) e^{2\pi i \ell x/b}
\end{align*}
As
\begin{align*}
    \tilde{c}_\ell(g) - c_\ell(g) 
    = \sum_{\substack{m=-\infty \\ m \ne 0}}^\infty c_{\ell+(bn)m}(g),
\end{align*}
the approximation error reads
\begin{align} \label{eq:error1}
    &g(x)-\tau_{n}(g)(x)
    = \sum_{\ell=-nb/2}^{(nb/2) -1} \sum_{\substack{m=-\infty \\ m \ne 0}}^{\infty} c_{\ell+(nb)m}(g) \left( e^{2\pi i n m x} - 1 \right) e^{2\pi i \ell x/b} 
\end{align}

In \cref{lemma_interp_error}, we obtain a generic bound on such interpolation errors.

\begin{lemma}\label{lemma_interp_error}
Let $b\in\mathbb{Q}$ with $b>1$ and
$g :\mathbb{R} \to \mathbb{R}$ be a $b$-periodic function of the form
\begin{align*}
  g(x)& = \begin{cases}
    g_1(x), x\in [0 , 1]\\
    g_2(x), x\in (1 , b)
\end{cases}  
\end{align*}
where $g_1 \in C^{(d+1)}[0,1]$ and $g_2 \in C^{(d+1)}[1,b]$.
Then, for $d\in\mathbb{N}$,
there exists a positive constant $M$ such that 
\begin{align*}
    &\Vert \tau_{n}(g)-g\Vert_{\infty,[0,1]} \le
        M \left( \sum_{k=0}^d \left(  \left| g_1^{(k)}(0+) -g_2^{(k)}(b-) \right| +\left| g_2^{(k)}(1+)-g_1^{(k)}(1-)\right|\right)n^{-k}+n^{-d} \right)
\end{align*}
for all $n\in\mathbb{N}_b$.
where 
\begin{align*}
    \varphi^{(k)}(0+) = \lim_{\substack{h \to 0 \\ h > 0}} \varphi^{(k)}(h), \quad \varphi^{(k)}(1\pm) = \lim_{\substack{h \to 0 \\ h > 0}} \varphi^{(k)}(1\pm h), \quad \varphi^{(k)}(b-) = \lim_{\substack{h \to 0 \\ h > 0}} \varphi^{(k)}(b- h).
\end{align*}
\end{lemma}
\begin{proof}

For $\ell\ne 0$, 
we have
\begin{align} \label{eq:clg}
    c_{\ell}(g)& =  \frac{1}{2 \pi i \ell}\left(\left( g_1(0+) -g_2(b-) \right) +\left( g_2(1+)-g_1(1-)\right) e^{- 2 \pi i \ell/b}  \right) \\ 
       &+ \frac{1}{b}\sum\limits_{k = 1}^{d} \left( \frac{b}{2\pi i \ell}\right)^{k+1} \left( \left(g_1^{(k)}(0+) - g_2^{(k)}(b-) \right) + \left(g_2^{(k)}(1+) -  g_1^{(k)}(1-)\right) e^{- 2\pi i \ell/b} \right) \nonumber \\
     &+\frac{1}{b}\left( \frac{b}{2\pi i \ell}\right)^{d+1} \left( \int_0^1 g_1^{(d+1)}(x) e^{-2 \pi i \ell x/b}\, dx +\int_1^b g_2^{(d+1)}(x) e^{-2 \pi i \ell x/b}\, dx \right) \nonumber
\end{align}

Using \cref{eq:error1}, the approximation error

\begin{align*}
    &|g(x)-\tau_{n}(g)(x)| \le
     2 \sum_{\ell=-nb/2}^{(nb/2) -1} \sum_{\substack{m=-\infty \\ m \ne 0}}^{\infty} |c_{\ell+(nb)m}(g)|\le \\
  & \frac{2}{\pi} \left( \left|g_1(0+)-g_2(b-) \right| + \left| g_2(1+)-g_1(1-)\right| \right) \sum_{m=1}^{\infty} \frac{1}{4m^2-1} +\\
    & \frac{2}{b} \sum_{k=1}^d \left(\frac{b}{\pi }\right)^{k+1} \left(\frac{1}{nb}\right)^{k} \left( \left|  g_1^{(k)}(0+)-g_2^{(k)}(b-)\right| + \left| g_2^{(k)}(1+)-g_1^{(k)}(1-)\right| \right) \sum_{m = 1}^{\infty}  \left(\frac{1}{2m-1}\right)^{k+1}+\\
    & \frac{2}{b}\left(\frac{b}{\pi}\right)^{d+1}\left(\frac{1}{nb}\right)^{d}  \left( \left\Vert g_1^{(d+1)} \right\Vert_{1,[0,1]} + \left\Vert g_2^{(d+1)} \right\Vert_{1,[1,b]} \right)  \sum_{m = 1}^{\infty}\left(\frac{1}{2m-1}\right)^{d+1},
\end{align*}
and, hence, the result follows.
\end{proof}

The \cref{lemma_interp_error} immediately implies
\begin{align} \label{eq:error_ref}
    \Vert f^c_{\text{ref}} - \tau_{n}(f^c_{\text{ref}})\Vert_{\infty,[0,1]} \le M\left( \left| f^{(d)}(0) -(p^d_{\text{ref}})^{(d)}(b) \right| +\left| (p^d_{\text{ref}})^{(d)}(1)-f^{(d)}(1)\right| + 1\right)n^{-d}
\end{align}
as $f^{(m)}(1)=\left(p^d_{\text{ref}}\right)^{(m)}(1)$ and $f^{(m)}(0)=\left(p^d_{\text{ref}}\right)^{(m)}(b)$ for $m=0,\ldots,d-1$. Before we estimate the other term in \cref{eq:error0} study, in the following lemma,
 how well the projection operator defined in \cref{ProjectionOnGramPoly} approximates a given function and its derivatives. 

\begin{table} [!t] 
	\begin{tabularx}{0.99\textwidth}{  >{\setlength\hsize{0.4\hsize}\centering}X | >{\setlength\hsize{1.3\hsize}\centering}X | >{\setlength\hsize{0.9\hsize}\centering}X | >{\setlength\hsize{1.3\hsize}\centering}X | >{\setlength\hsize{0.9\hsize}\centering}X | >{\setlength\hsize{1.3\hsize}\centering}X |
			>{\setlength\hsize{0.9\hsize}\centering}X  }
		\hline
		\multirow{2}{\hsize}{$n$} &  \multicolumn{2}{>{\setlength\hsize{2.2\hsize}\centering}X |}{$d = 3$} &
		\multicolumn{2}{>{\setlength\hsize{2.2\hsize}\centering}X |}{$d = 4$} &
		\multicolumn{2}{>{\setlength\hsize{2.2\hsize}\centering}X }{$d = 5$} \tabularnewline
		\cline{2-7}
		& $e_n$ & $noc_n$ &  $e_n$  & $noc_n$  & $e_n$ & $noc_n$ \tabularnewline
		\hline
        \hline
$2^{6}$ & $ 1.74\times 10^{-04} $ & ---  & $ 5.03\times 10^{-05} $ & ---& $ 2.74\times 10^{-05} $ & ---
 \tabularnewline
		\hline
$2^{7}$ & $ 2.31\times 10^{-05} $ & $ 2.92 $ & $ 1.17\times 10^{-06} $ & $ 5.42 $ & $ 1.31\times 10^{-06} $ & $ 4.39 $
\tabularnewline
		\hline
$2^{8}$ & $ 2.90\times 10^{-06} $ & $ 2.99 $  & $ 4.49\times 10^{-08} $ & $ 4.71 $  & $ 4.04\times 10^{-08} $ & $ 5.02 $
  \tabularnewline
		\hline
$2^{9}$ & $ 3.62\times 10^{-07} $ & $  3.00 $ & $ 2.83\times 10^{-09} $ & $ 3.99 $  & $ 1.19\times 10^{-09} $ & $ 5.08 $
\tabularnewline
		\hline
$2^{10}$ & $ 4.51\times 10^{-08} $ & $  3.00 $ & $ 1.77\times 10^{-10} $ & $  4.00 $  & $ 3.59\times 10^{-11} $ & $ 5.06 $
\tabularnewline
		\hline
$2^{11}$ & $ 5.62\times 10^{-09} $ & $  3.00 $  & $ 1.11\times 10^{-11} $ & $  4.00 $ & $ 1.09\times 10^{-12} $ & $ 5.04 $
\tabularnewline
		\hline
$2^{12}$ & $ 7.02\times 10^{-10} $ & $  3.00 $  & $ 7.09\times 10^{-13} $ & $ 3.97 $ & $ 8.09\times 10^{-14} $ & $ 3.75 $
    \tabularnewline
    \hline
	\end{tabularx}
	\caption{Convergence study for approximations of $f(x) = \exp\left(\sin(5.4 \pi t- 2.7 \pi) -\cos(2 \pi t) \right)$ using the derivative approximations using Gram polynomials of degree $d = 3,4,5$ with periodic extension length $ b = 2$.}
	\label{ex 1.1}
\end{table}

\begin{lemma}\label{lemma_projection}
    For $d\in \mathbb{N},\ d>1$,  the projection operators
    \begin{align*}
        P_d^L(\psi)(x)= \sum_{\ell = 0}^{d-1} \langle \psi, p_{\ell}^L\rangle_L p_{\ell}^L(x), \quad P_d^R(\psi)(x)= \sum_{\ell = 0}^{d-1} \langle \psi, p_{\ell}^R\rangle_R p_{\ell}^R(x),
    \end{align*}

   respectively satisfy 
  \begin{align*}
      \Vert (P_d^L\psi)^{(m)} - \psi^{(m)} \Vert_{\infty, [0, (d-1)h]} &\leq M h^{d-m}, \quad \Vert (P_d^R\psi)^{(m)} - \psi^{(m)} \Vert_{\infty, [1-(d-1)h,1]} \leq M h^{d-m}.
  \end{align*}
  for some constant $M>0$, for $0\leq m \leq d-1$ and for $\psi\in C^{d}\left( [0, 1]\right)$.
\end{lemma}
\begin{proof}
We present the arguments for the left projection operator, and the corresponding result for the right counterpart follows similarly.

\begin{table} [!t] 
	\begin{tabularx}{0.99\textwidth}{  >{\setlength\hsize{0.4\hsize}\centering}X | >{\setlength\hsize{1.3\hsize}\centering}X | >{\setlength\hsize{0.9\hsize}\centering}X | >{\setlength\hsize{1.3\hsize}\centering}X | >{\setlength\hsize{0.9\hsize}\centering}X | >{\setlength\hsize{1.3\hsize}\centering}X |
			>{\setlength\hsize{0.9\hsize}\centering}X  }
		\hline
		\multirow{2}{\hsize}{$n$} &  \multicolumn{2}{>{\setlength\hsize{2.2\hsize}\centering}X |}{$d = 3$} &
		\multicolumn{2}{>{\setlength\hsize{2.2\hsize}\centering}X |}{$d = 4$} &
		\multicolumn{2}{>{\setlength\hsize{2.2\hsize}\centering}X }{$d = 5$} \tabularnewline
		\cline{2-7}
		& $e_n$ & $noc_n$ &  $e_n$  & $noc_n$  & $e_n$ & $noc_n$ \tabularnewline
		\hline
        \hline
$2^{6}$ & $ 2.47\times 10^{-04} $ & ---  & $ 1.93\times 10^{-04} $ & ---  & $ 3.46\times 10^{-04} $ & ---
 \tabularnewline
		\hline
$2^{7}$ & $ 2.31\times 10^{-05} $ & $ 3.42 $   & $ 1.51\times 10^{-05} $ & $ 3.67 $  & $ 6.05\times 10^{-06} $ & $ 5.84 $
\tabularnewline
		\hline
$2^{8}$ & $ 2.31\times 10^{-06} $ & $ 3.32 $  & $ 9.62\times 10^{-07} $ & $ 3.98 $  & $ 8.89\times 10^{-08} $ & $ 6.09 $
  \tabularnewline
		\hline
$2^{9}$ & $ 3.00\times 10^{-07} $ & $ 2.95 $ & $ 5.95\times 10^{-08} $ & $ 4.02 $ & $ 3.02\times 10^{-09} $ & $ 4.88 $
\tabularnewline
		\hline
$2^{10}$ & $ 3.97\times 10^{-08} $ & $ 2.92 $  & $ 3.68\times 10^{-09} $ & $ 4.02 $ & $ 1.16\times 10^{-10} $ & $ 4.71 $ 
\tabularnewline
		\hline
$2^{11}$ & $ 5.08\times 10^{-09} $ & $ 2.96 $  & $ 2.28\times 10^{-10} $ & $ 4.01 $   & $ 3.93\times 10^{-12} $ & $ 4.88 $ 
\tabularnewline
		\hline
$2^{12}$ & $ 6.44\times 10^{-10} $ & $ 2.98 $ & $ 1.40\times 10^{-11} $ & $ 4.03 $  & $ 1.28\times 10^{-13} $ & $ 4.94 $ 
\tabularnewline
		\hline 
	\end{tabularx}
 \caption{Convergence study for approximations of $f(x) = \exp\left(\sin(5.4 \pi t- 2.7 \pi) -\cos(2 \pi t) \right)$ using the derivative approximations using Gram polynomials of degree $d = 3,4,5$ with periodic extension length $ b = 1.0625$.}
	\label{ex 1.2}
\end{table}
        
We begin by noting that, using \cref{Gram_Schmidt}, we can express monomials $x^k$ as a linear combination of Gram polynomials as 
\begin{align*}
 x^{k}
 = h^{k}\sum_{\ell = 0}^{k} \beta_{\ell}  p_\ell^L(x)
\end{align*}
Therefore, the Taylor expansion around $x=0$ of $\psi$ can be expressed as
\begin{align}\label{psi}
    \psi(x)& = \sum\limits_{k = 0}^{d-1} \frac{\psi^{(k)}(0)}{k!} \left( h^{k} \sum_{\ell = 0}^{k} \beta_{\ell} p_{\ell}^L(x) \right) + R_d\left( x\right)= \sum_{\ell =0}^{d-1}\left(\sum_{k=\ell}^{d-1} \beta_{\ell} h^{k} \frac{\psi^{(k)}(0)}{k!} \right) p_{\ell}^L(x) + R_d\left( x\right)
\end{align}
where the remainder $R_d$ is
\begin{align}\label{R_d}
    R_d(x)& = \frac{1}{(d-1)!}\int_0^x ( x-t)^{d-1} \psi^{(d)}(t)\, dt
\end{align}

Now, for $0\leq m \leq d-1$, using \cref{psi}, we have
\begin{align*}
    \langle \psi, p_m^L\rangle_L 
    = \sum_{k=m}^{d-1} \beta_{m} h^{k} \frac{\psi^{(k)}(0)}{k!} + \langle R_d, p_m^L \rangle,
\end{align*}
and, therefore, the projection of $\psi$ reads
\begin{align*}
    \left(P_d\psi\right)(x) 
    &= \sum_{\ell =0}^{d-1} \left( \sum_{k=\ell}^{d-1} \beta_{\ell} h^{k} \frac{\psi^{(k)}(0)}{k!} \right) p_{\ell}^L(x) + \sum_{\ell =0}^{d-1}  \langle R_d, p_\ell^L \rangle p_{\ell}^L(x)
    \end{align*}
Thus, for $0\leq m \leq d-1$, we have

\begin{align*}
   \psi^{(m)}(x) -  \left(P_d\psi\right)^{(m)}(x)
   &= \left(R_d\right)^{(m)}\left( x\right) -  \sum_{\ell =m}^{d-1}  \langle R_d(x), p_\ell^L \rangle (p_{\ell}^L)^{(m)}(x)
\end{align*}
Finally, using \ref{R_d}, we have
\begin{align*}
    &\Vert \psi^{(m)} - \left(P_d\psi\right)^{(m)} \Vert_{\infty, [0, (d-1)h]} \leq\\
    &\sup_{x\in [0, (d-1)h]} \left( \left| \frac{1}{(d-1)!}\int_0^x (x-t)^{d-m -1}\psi^{(d)}(t) dt \right| + \left| \sum_{\ell =m}^{d-1}  \left(\sum_{j =0}^{d-1} R_d(j h) p_{\ell}^L(jh) \right)(p_{\ell}^L)^{(m)}(x) \right| \right) \leq \\
    & \frac{1}{(d-1)!}\Vert \psi^{(d)}\Vert_{\infty} \left( \frac{(d-1)^{d-m}}{d-m}+ (d-1)^{d} \left( \sum_{\ell = m}^{d-1} \Vert \tilde{p}_{\ell}^L\Vert_{\infty, [0, d-1]} \Vert(\tilde{p}_{\ell}^L)^{(m)} \Vert_{[0,d-1]} \right)\right) h^{d-m}
\end{align*}
where we have used the fact that $(p_{\ell}^L)(x) = (\tilde{p}_{\ell}^L)(x/h)$ and $\Vert (p_{\ell}^L)^{(m)} \Vert_{[0,(d-1)h]} = h^{-m}\Vert(\tilde{p}_{\ell}^L)^{(m)} \Vert_{[0,d-1]}$. Thus,
the result follows.

\end{proof}

\begin{table} [!t] 
	\begin{tabularx}{0.99\textwidth}{  >{\setlength\hsize{0.4\hsize}\centering}X | >{\setlength\hsize{1.3\hsize}\centering}X | >{\setlength\hsize{0.9\hsize}\centering}X | >{\setlength\hsize{1.3\hsize}\centering}X | >{\setlength\hsize{0.9\hsize}\centering}X | >{\setlength\hsize{1.3\hsize}\centering}X |
			>{\setlength\hsize{0.9\hsize}\centering}X  }
		\hline
		\multirow{2}{\hsize}{$n$} &  \multicolumn{2}{>{\setlength\hsize{2.2\hsize}\centering}X |}{$d = 3$} &
		\multicolumn{2}{>{\setlength\hsize{2.2\hsize}\centering}X |}{$d = 4$} &
		\multicolumn{2}{>{\setlength\hsize{2.2\hsize}\centering}X }{$d = 5$} \tabularnewline
		\cline{2-7}
		& $e_n$ & $noc_n$ &  $e_n$  & $noc_n$  & $e_n$ & $noc_n$ \tabularnewline
		\hline
        \hline
$2^{6}$ & $ 8.58\times 10^{-07} $ & --- & $ 9.96\times 10^{-08} $ & --- & $ 3.58\times 10^{-09} $ & --- 
 \tabularnewline
		\hline
$2^{7}$ & $ 1.08\times 10^{-07} $ & $ 2.98 $  & $ 6.17\times 10^{-09} $ & $ 4.01 $ & $ 1.18\times 10^{-10} $ & $ 4.92 $ 
\tabularnewline
		\hline
$2^{8}$ & $ 1.36\times 10^{-08} $ & $ 2.99 $ & $ 3.84\times 10^{-10} $ & $ 4.01 $ & $ 3.79\times 10^{-12} $ & $ 4.96 $
  \tabularnewline
		\hline
$2^{9}$ & $ 1.71\times 10^{-09} $ & $  3.00 $  & $ 2.40\times 10^{-11} $ & $  4.00 $  & $ 1.22\times 10^{-13} $ & $ 4.95 $
\tabularnewline
		\hline
$2^{10}$ & $ 2.14\times 10^{-10} $ & $  3.00 $ & $ 1.50\times 10^{-12} $ & $  4.00 $  & $ 2.39\times 10^{-14} $ & $ 2.36 $ 
\tabularnewline
		\hline
$2^{11}$ & $ 2.67\times 10^{-11} $ & $  3.00 $ & $ 1.10\times 10^{-13} $ & $ 3.77 $ & $ 2.48\times 10^{-14} $ & $ -0.06 $ 
\tabularnewline
		\hline
$2^{12}$ & $ 3.34\times 10^{-12} $ & $  3.00 $ & $ 5.96\times 10^{-14} $ & $ 0.88 $  & $ 2.06\times 10^{-14} $ & $ 0.27 $
  \tabularnewline
		\hline
	\end{tabularx}
 \caption{Convergence study for approximations of $f(x) = \exp(x)$ using the derivative approximations using Gram polynomials of degree $d = 3,4,5$ with periodic extension length $ b = 2$.}
	\label{ex 2.1}
\end{table}

\begin{table} [!b] 
	\begin{tabularx}{0.99\textwidth}{  >{\setlength\hsize{0.4\hsize}\centering}X | >{\setlength\hsize{1.3\hsize}\centering}X | >{\setlength\hsize{0.9\hsize}\centering}X | >{\setlength\hsize{1.3\hsize}\centering}X | >{\setlength\hsize{0.9\hsize}\centering}X | >{\setlength\hsize{1.3\hsize}\centering}X |
			>{\setlength\hsize{0.9\hsize}\centering}X  }
		\hline
		\multirow{2}{\hsize}{$n$} &  \multicolumn{2}{>{\setlength\hsize{2.2\hsize}\centering}X |}{$d = 3$} &
		\multicolumn{2}{>{\setlength\hsize{2.2\hsize}\centering}X |}{$d = 4$} &
		\multicolumn{2}{>{\setlength\hsize{2.2\hsize}\centering}X }{$d = 5$} \tabularnewline
		\cline{2-7}
		& $e_n$ & $noc_n$ &  $e_n$  & $noc_n$  & $e_n$ & $noc_n$ \tabularnewline
		\hline
        \hline
$2^{6}$ & $ 8.09\times 10^{-04} $ & ---  & $ 2.31\times 10^{-03} $ & --- & $ 4.13\times 10^{-03} $ & ---
 \tabularnewline
		\hline
$2^{7}$ & $ 1.23\times 10^{-04} $ & $ 2.72 $  & $ 1.87\times 10^{-04} $ & $ 3.63 $  & $ 6.45\times 10^{-05} $ & $  6 $ 
\tabularnewline
		\hline
$2^{8}$ & $ 2.11\times 10^{-05} $ & $ 2.54 $ & $ 1.20\times 10^{-05} $ & $ 3.97 $ & $ 7.84\times 10^{-07} $ & $ 6.36 $
  \tabularnewline
		\hline
$2^{9}$ & $ 3.05\times 10^{-06} $ & $ 2.79 $ & $ 7.42\times 10^{-07} $ & $ 4.01 $ & $ 3.96\times 10^{-08} $ & $ 4.31 $
\tabularnewline
		\hline
$2^{10}$ & $ 4.14\times 10^{-07} $ & $ 2.88 $ & $ 4.59\times 10^{-08} $ & $ 4.02 $ & $ 1.53\times 10^{-09} $ & $ 4.7 $
\tabularnewline
		\hline
$2^{11}$ & $ 5.38\times 10^{-08} $ & $ 2.94 $ & $ 2.85\times 10^{-09} $ & $ 4.01 $ & $ 5.23\times 10^{-11} $ & $ 4.87 $
\tabularnewline
		\hline
$2^{12}$ & $ 6.78\times 10^{-09} $ & $ 2.99 $ & $ 1.74\times 10^{-10} $ & $ 4.03 $ & $ 1.70\times 10^{-12} $ & $ 4.94 $
  \tabularnewline
		\hline
	\end{tabularx}
  \caption{Convergence study for approximations of $f(x) = \exp(x)$ using the derivative approximations using Gram polynomials of degree $d = 3,4,5$ with periodic extension length $ b = 1.0625$.}
	\label{ex 2.2}
\end{table}

To estimate $(f^c - f^c_{\text{ref}}) - \tau_{n}(f^c-f^c_{\text{ref}})$ in \cref{eq:error0}, we apply \cref{lemma_interp_error} to the function $g = f^c- f^c_{\text{ref}}$. In this context, as $g \equiv 0$ on $[0,1]$, the estimate simplifies to
\begin{align} \label{eq:error2}
    &\Vert \tau_{n}(f^c - f^c_{\text{ref}})-(f^c - f^c_{\text{ref}})\Vert_{\infty,[0,1]} \le
        M \left( \sum_{k=0}^d \left(  \left| (p - p^d_{\text{ref}})^{(k)}(b) \right| +\left| (p - p^d_{\text{ref}})^{(k)}(1)\right|\right)n^{-k}+n^{-d} \right)
\end{align}
To estimate $\left| (p - p^d_{\text{ref}})^{(k)}(x) \right|$, a crucial step is in rewriting the extension $p$ as 
 \begin{align*}
 p(x) &= \sum_{\ell = 0}^{d-1} \langle f,p_\ell^L \rangle_{L}\,  \left(\sum\limits_{m = 0}^{\ell} ( p_{\ell}^{L})^{(m)}\left(0\right) p_{m}^{b, 1}(x)\right)  + \sum_{\ell = 0}^{d-1} \langle f,p_\ell^R \rangle_{R}\,  \left(\sum\limits_{m = 0}^{\ell} ( p_{\ell}^{R})^{(m)}\left(0\right) p_{m}^{1, b}(x)\right)   \\
   &= \sum\limits_{m = 0}^{d-1} (P_d^L f)^{(m)} (0) p_{m}^{b, 1}(x) + \sum\limits_{m = 0}^{d-1} (P_d^R f)^{(m)} (1) p_{m}^{1, b}(x)
 \end{align*}
 As a result, we have
   \begin{align*}
     (p - p^d_{\text{ref}})^{(m)}(x) & = \sum\limits_{\ell = 0}^{d-1} \left[ (P_d^L f)^{(\ell)} (0) - f^{(\ell)}(0) \right](p_{\ell}^{b, 1})^{(m)}(x) + \sum\limits_{\ell = 0}^{d-1} \left[ (P_d^R f)^{(\ell)} (1) - f^{(\ell)}(1) \right] (p_{\ell}^{1, b})^{(m)}(x).
 \end{align*}
 Thus, for all $x\in [1,b]$, we have
    \begin{align} \label{eq:error3}
     | (p - p^d_{\text{ref}})^{(m)}(x) |& \le \sum\limits_{\ell = 0}^{d-1} \left[ (P_d^L f)^{(\ell)} (0) - f^{(\ell)}(0) \right]\Vert p_{\ell}^{b, 1})^{(m)}\Vert_{\infty,[1,b]} + \sum\limits_{\ell = 0}^{d-1} \left[ (P_d^R f)^{(\ell)} (1) - f^{(\ell)}(1) \right] \Vert p_{\ell}^{1, b})^{(m)}\Vert_{\infty,[1,b]}.
 \end{align}
 Combining the estimates in \cref{lemma_projection} and \cref{eq:error3}, we get
     \begin{align*} 
     | (p - p^d_{\text{ref}})^{(m)}(x) |& \le Mh^{d-m} \left( \sum\limits_{\ell = 0}^{d-1} \Vert p_{\ell}^{b, 1})^{(m)}\Vert_{\infty,[1,b]} + \sum\limits_{\ell = 0}^{d-1}  \Vert p_{\ell}^{1, b})^{(m)}\Vert_{\infty,[1,b]} \right),
 \end{align*}
 which, together with \cref{eq:error2}, leads to
\begin{align} \label{eq:error4}
    &\Vert \tau_{n}(f^c - f^c_{\text{ref}})-(f^c - f^c_{\text{ref}})\Vert_{\infty,[0,1]} \le
        M \left( \sum_{k=0}^d 2M\left( \sum\limits_{\ell = 0}^{d-1} \Vert p_{\ell}^{b, 1})^{(m)}\Vert_{\infty,[1,b]} + \sum\limits_{\ell = 0}^{d-1}  \Vert p_{\ell}^{1, b})^{(m)}\Vert_{\infty,[1,b]} \right)+1 \right)n^{-d}
\end{align}

\begin{table} [!t] 
	\begin{tabularx}{0.99\textwidth}{  >{\setlength\hsize{0.4\hsize}\centering}X | >{\setlength\hsize{1.3\hsize}\centering}X | >{\setlength\hsize{0.9\hsize}\centering}X | >{\setlength\hsize{1.3\hsize}\centering}X | >{\setlength\hsize{0.9\hsize}\centering}X | >{\setlength\hsize{1.3\hsize}\centering}X |
			>{\setlength\hsize{0.9\hsize}\centering}X  }
		\hline
		\multirow{2}{\hsize}{$n$} &  \multicolumn{2}{>{\setlength\hsize{2.2\hsize}\centering}X |}{$k = 50$} &
		\multicolumn{2}{>{\setlength\hsize{2.2\hsize}\centering}X |}{$k = 100$} &
		\multicolumn{2}{>{\setlength\hsize{2.2\hsize}\centering}X }{$k = 200$} \tabularnewline
		\cline{2-7}
		& $e_n$ & $noc_n$ &  $e_n$  & $noc_n$  & $e_n$ & $noc_n$ \tabularnewline
		\hline
        \hline
$2^{6}$ & $ 1.04\times 10^{-02} $ & --- & $ 3.19\times 10^{-01} $ & ---  & $ 1.32e+00 $ & --- 
 \tabularnewline
		\hline
 $2^{7}$ & $ 2.27\times 10^{-04} $ & $ 5.52 $  & $ 1.03\times 10^{-02} $ & $ 4.95 $ & $ 3.28\times 10^{-01} $ & $ 2.01 $ 
\tabularnewline
		\hline
$2^{8}$ & $ 1.35\times 10^{-06} $ & $ 7.39 $  & $ 4.37\times 10^{-04} $ & $ 4.56 $ & $ 2.84\times 10^{-02} $ & $ 3.53 $ 
  \tabularnewline
		\hline
$2^{9}$ & $ 6.98\times 10^{-09} $ & $ 7.6 $ & $ 5.05\times 10^{-06} $ & $ 6.44 $ & $ 5.47\times 10^{-04} $ & $ 5.7 $ 
\tabularnewline
		\hline
 $2^{10}$ & $ 7.96\times 10^{-11} $ & $ 6.45 $ & $ 6.98\times 10^{-08} $ & $ 6.18 $ & $ 1.97\times 10^{-05} $ & $ 4.8 $ 
\tabularnewline
		\hline
$2^{11}$ & $ 3.94\times 10^{-12} $ & $ 4.34 $ & $ 1.28\times 10^{-09} $ & $ 5.77 $  & $ 4.60\times 10^{-07} $ & $ 5.42 $
\tabularnewline
		\hline
$2^{12}$ & $ 1.78\times 10^{-12} $ & $ 1.15 $ & $ 6.28\times 10^{-11} $ & $ 4.35 $ & $ 1.18\times 10^{-08} $ & $ 5.29 $
  \tabularnewline
		\hline
	\end{tabularx}
	\caption{Convergence study for approximations of $f(x) = \exp(-\cos(k x))$ with $k = 50, 100$ and $200$ using the derivative approximations using Gram polynomials of degree $d=4$ with periodic extension length $ b = 2$}
	\label{ex 3.1}
\end{table}

\begin{table} [!t] 
	\begin{tabularx}{0.99\textwidth}{  >{\setlength\hsize{0.4\hsize}\centering}X | >{\setlength\hsize{1.3\hsize}\centering}X | >{\setlength\hsize{0.9\hsize}\centering}X | >{\setlength\hsize{1.3\hsize}\centering}X | >{\setlength\hsize{0.9\hsize}\centering}X | >{\setlength\hsize{1.3\hsize}\centering}X |
			>{\setlength\hsize{0.9\hsize}\centering}X  }
		\hline
		\multirow{2}{\hsize}{$n$} &  \multicolumn{2}{>{\setlength\hsize{2.2\hsize}\centering}X |}{$k = 50$} &
		\multicolumn{2}{>{\setlength\hsize{2.2\hsize}\centering}X |}{$k = 100$} &
		\multicolumn{2}{>{\setlength\hsize{2.2\hsize}\centering}X }{$k = 200$} \tabularnewline
		\cline{2-7}
		& $e_n$ & $noc_n$ &  $e_n$  & $noc_n$  & $e_n$ & $noc_n$ \tabularnewline
		\hline
        \hline
$2^{6}$ & $ 1.21\times 10^{-02} $ & ---   & $ 3.38\times 10^{-01} $ & ---   & $ 1.32e+00 $ & ---
 \tabularnewline
		\hline
 $2^{7}$ & $ 2.44\times 10^{-04} $ & $ 5.63 $ & $ 1.02\times 10^{-02} $ & $ 5.06 $ & $ 3.40\times 10^{-01} $ & $ 1.95 $ 
\tabularnewline
		\hline
$2^{8}$ & $ 1.50\times 10^{-06} $ & $ 7.35 $ & $ 4.39\times 10^{-04} $ & $ 4.53 $ & $ 2.78\times 10^{-02} $ & $ 3.62 $ 
  \tabularnewline
		\hline
$2^{9}$ & $ 1.09\times 10^{-08} $ & $ 7.11 $ & $ 5.06\times 10^{-06} $ & $ 6.44 $  & $ 5.46\times 10^{-04} $ & $ 5.67 $
\tabularnewline
		\hline
$2^{10}$ & $ 3.83\times 10^{-10} $ & $ 4.83 $  & $ 7.00\times 10^{-08} $ & $ 6.18 $ & $ 1.96\times 10^{-05} $ & $ 4.8 $
\tabularnewline
		\hline
$2^{11}$ & $ 1.28\times 10^{-11} $ & $ 4.90 $ & $ 1.29\times 10^{-09} $ & $ 5.77 $ & $ 4.60\times 10^{-07} $ & $ 5.42 $
\tabularnewline
		\hline
$2^{12}$ & $ 4.14\times 10^{-13} $ & $ 4.95 $  & $ 2.94\times 10^{-11} $ & $ 5.45 $ & $ 1.18\times 10^{-08} $ & $ 5.29 $ 
  \tabularnewline
		\hline
	\end{tabularx}
 \caption{Convergence study for approximations of $f(x) = \exp(-\cos(k x))$ with $k = 50, 100$ and $200$ using the derivative approximations using Gram polynomials of degree $d=4$ with periodic extension length $ b = 1.0625$.}
	\label{ex 3.2}
\end{table}

\begin{table} [!t] 
	\begin{tabularx}{0.99\textwidth}{  >{\setlength\hsize{0.4\hsize}\centering}X | >{\setlength\hsize{1.3\hsize}\centering}X | >{\setlength\hsize{0.9\hsize}\centering}X | >{\setlength\hsize{1.3\hsize}\centering}X | >{\setlength\hsize{0.9\hsize}\centering}X | >{\setlength\hsize{1.3\hsize}\centering}X |
			>{\setlength\hsize{0.9\hsize}\centering}X  }
		\hline
		\multirow{2}{\hsize}{$n$} &  \multicolumn{2}{>{\setlength\hsize{2.2\hsize}\centering}X |}{$\epsilon = 1$} &
		\multicolumn{2}{>{\setlength\hsize{2.2\hsize}\centering}X |}{$\epsilon = 0.1$} &
		\multicolumn{2}{>{\setlength\hsize{2.2\hsize}\centering}X }{$\epsilon = 0.01$} \tabularnewline
		\cline{2-7}
		& $e_n$ & $noc_n$ &  $e_n$  & $noc_n$  & $e_n$ & $noc_n$ \tabularnewline
		\hline
        \hline
$2^{6}$ & $ 2.26\times 10^{-09} $ & --- & $ 2.66\times 10^{-07} $ & --- & $ 2.06\times 10^{-01} $ & ---
 \tabularnewline
		\hline
 $2^{7}$ & $ 6.70\times 10^{-11} $ & $ 5.08 $ & $ 6.73\times 10^{-09} $ & $ 5.31 $ & $ 3.02\times 10^{-02} $ & $ 2.77 $
\tabularnewline
		\hline
$2^{8}$ & $ 2.03\times 10^{-12} $ & $ 5.04 $ & $ 1.89\times 10^{-10} $ & $ 5.15 $ & $ 5.98\times 10^{-04} $ & $ 5.66 $
  \tabularnewline
		\hline
$2^{9}$ & $ 6.85\times 10^{-14} $ & $ 4.89 $ & $ 5.61\times 10^{-12} $ & $ 5.08 $  & $ 1.92\times 10^{-07} $ & $ 11.61 $
\tabularnewline
		\hline
 $2^{10}$ & $ 1.70\times 10^{-14} $ & $ 2.01 $ & $ 1.71\times 10^{-13} $ & $ 5.03 $ & $ 2.22\times 10^{-14} $ & $ 23.04 $
\tabularnewline
		\hline
$2^{11}$ & $ 1.67\times 10^{-14} $ & $ 0.03 $ & $ 8.63\times 10^{-15} $ & $ 4.31 $ & $ 8.91\times 10^{-15} $ & $ 1.32 $
\tabularnewline
		\hline
$2^{12}$ & $ 1.91\times 10^{-14} $ & $ -0.20 $ & $ 5.28\times 10^{-15} $ & $ 0.71 $ & $ 8.37\times 10^{-15} $ & $ 0.09 $ 
  \tabularnewline
		\hline
	\end{tabularx} 
  \caption{Convergence study for approximations of $f_{\epsilon}(x)= ((x-1/3)^2+ \epsilon^2)^{-1}$ with $\epsilon = 1, ~0.1$ and $0.01$ using the derivative approximations using Gram polynomials of degree $d=4$ with periodic extension length $ b = 2$.}
	\label{ex 4.1}
\end{table}

\begin{proof}[Proof of \cref{theorem: main}]
The result directly follows from \cref{eq:error0} and the estimates \cref{eq:error_ref}, and \cref{eq:error4}.

\end{proof}

\section{Numerical Implementation} \label{sec: numerics}

To show that the theoretical convergence rates mentioned above are reached, we will now go over a few numerical experiments. In pursuit of this goal, we examine the issue of utilizing functional data on a uniform grid of size $n$ to approximate a function $f(x)$ on $[0, 1]$. The relative approximation error en, which we obtain, is noted as
\begin{align*}
    e_n &= \max_{0\leq j \leq N}\vert\tau_nf(z_j) - f(z_j) \vert/ \max_{0\leq j \leq N}\vert f(z_j) \vert 
\end{align*}
where $N = 2^{15}$ and $z_j = j/N$ are the evaluation points on a large uniform grid where approximate and exact values are compared.

In the first set of experiments, we study the effect of $b$ and $d$ on the rate of convergence as $n$ increases.
The results in table \ref{ex 1.1} and table \ref{ex 1.2} for a smooth function $f(x) = \exp\left(\sin(5.4 \pi x- 2.7 \pi )-\cos(2 \pi x) \right)$ shows that the numerical
rate of convergence indeed matches the theoretical rate $d$. In addition, we observe that when we reduce the domain of the periodic extension, denoted as $b$, the rates continue to be faithful to the assertions. Next, we look at $f(x) = \exp(x)$. The findings for different values of $d$ and $b$ demonstrate that the error obeys the theoretical assertion, as shown in tables \ref{ex 2.1} and \ref{ex 2.2}. Furthermore, as anticipated, the accuracy of approximations remains satisfactory even for functions that have substantial oscillations, as clearly seen in table \ref{ex 3.1} and table \ref{ex 3.2}.

Finally, we conclude this section by looking at the approximation quality of the proposed approach for $f_{\epsilon}(x)= ((x-1/3)^2+ \epsilon^2)^{-1}$ on $[0, 1]$ that has poles in the complex plane at $z = 1/3 \pm i$.
The Fourier continuation approximations employed in table \ref{ex 4.1} and table \ref{ex 4.2} correspond to the parameters $d = 5$ and $b = 2, ~ 1.0625$ respectively. Consequently, it is anticipated that these approximations will converge at a rate of $5$, as presented in the table, especially for $\epsilon = 1$ and $\epsilon = 0.1$. The results obtained when $\epsilon = 0.01$ demonstrate super algebraic convergence because of the comparatively small boundary data in comparison to the peak function value of $\epsilon^2$ at $x = 1/3$.

\begin{table} [!t] 
	\begin{tabularx}{0.99\textwidth}{  >{\setlength\hsize{0.4\hsize}\centering}X | >{\setlength\hsize{1.3\hsize}\centering}X | >{\setlength\hsize{0.9\hsize}\centering}X | >{\setlength\hsize{1.3\hsize}\centering}X | >{\setlength\hsize{0.9\hsize}\centering}X | >{\setlength\hsize{1.3\hsize}\centering}X |
			>{\setlength\hsize{0.9\hsize}\centering}X  }
		\hline
		\multirow{2}{\hsize}{$n$} &  \multicolumn{2}{>{\setlength\hsize{2.2\hsize}\centering}X |}{$\epsilon = 1$} &
		\multicolumn{2}{>{\setlength\hsize{2.2\hsize}\centering}X |}{$\epsilon = 0.1$} &
		\multicolumn{2}{>{\setlength\hsize{2.2\hsize}\centering}X }{$\epsilon = 0.01$} \tabularnewline
		\cline{2-7}
		& $e_n$ & $noc_n$ &  $e_n$  & $noc_n$  & $e_n$ & $noc_n$ \tabularnewline
		\hline
        \hline
$2^{6}$ & $ 1.29\times 10^{-03} $ & ---  & $ 3.06\times 10^{-04} $ & ---  & $ 2.06\times 10^{-01} $ & ---
 \tabularnewline
		\hline
 $2^{7}$ & $ 2.01\times 10^{-05} $ & $  6.00 $ & $ 4.77\times 10^{-06} $ & $  6.00 $   & $ 3.02\times 10^{-02} $ & $ 2.77 $
\tabularnewline
		\hline
$2^{8}$ & $ 2.53\times 10^{-07} $ & $ 6.32 $  & $ 6.09\times 10^{-08} $ & $ 6.29 $ & $ 5.98\times 10^{-04} $ & $ 5.66 $
  \tabularnewline
		\hline
$2^{9}$ & $ 1.26\times 10^{-08} $ & $ 4.32 $ & $ 3.02\times 10^{-09} $ & $ 4.33 $ & $ 1.92\times 10^{-07} $ & $ 11.61 $
\tabularnewline
		\hline
 $2^{10}$ & $ 4.85\times 10^{-10} $ & $ 4.70 $ & $ 1.16\times 10^{-10} $ & $ 4.70 $  & $ 1.27\times 10^{-12} $ & $ 17.20 $
\tabularnewline
		\hline
$2^{11}$ & $ 1.66\times 10^{-11} $ & $ 4.87 $  & $ 3.96\times 10^{-12} $ & $ 4.87 $  & $ 4.29\times 10^{-14} $ & $ 4.89 $
\tabularnewline
		\hline
$2^{12}$ & $ 5.39\times 10^{-13} $ & $ 4.94 $ & $ 1.29\times 10^{-13} $ & $ 4.94 $ & $ 7.09\times 10^{-15} $ & $ 2.60 $
  \tabularnewline
		\hline
	\end{tabularx} 
   \caption{Convergence study for approximations of $f_{\epsilon}(x)= ((x-1/3)^2+ \epsilon^2)^{-1}$ with $\epsilon = 1, ~0.1$ and $0.01$ using the derivative approximations using Gram polynomials of degree $d=4$ with periodic extension length $ b = 1.0625$.}
	\label{ex 4.2}
\end{table}


\end{document}